\documentclass[oneside,english,reqno]{amsart}
\usepackage{lmodern}
\usepackage[T1]{fontenc}
\usepackage[latin9]{inputenc}
\usepackage{babel}
\usepackage{amsthm}
\usepackage{amssymb}
\usepackage{graphicx}
\usepackage[unicode=true,pdfusetitle,
 bookmarks=true,bookmarksnumbered=false,bookmarksopen=false,
 breaklinks=false,pdfborder={0 0 1},backref=false,colorlinks=false]
 {hyperref}

\makeatletter

\providecommand{\tabularnewline}{\\}

\numberwithin{equation}{section}
\numberwithin{figure}{section}
\theoremstyle{plain}
\newtheorem{thm}{\protect\theoremname}[section]
  \theoremstyle{plain}
  \newtheorem{algorithm}[thm]{\protect\algorithmname}
  \theoremstyle{plain}
  \newtheorem{prop}[thm]{\protect\propositionname}
  \theoremstyle{remark}
  \newtheorem{rem}[thm]{\protect\remarkname}
  \theoremstyle{plain}
  \newtheorem{lem}[thm]{\protect\lemmaname}

\newcommand{\dom}{\mbox{\rm dom}}

\makeatother

  \providecommand{\algorithmname}{Algorithm}
  \providecommand{\lemmaname}{Lemma}
  \providecommand{\propositionname}{Proposition}
  \providecommand{\remarkname}{Remark}
\providecommand{\theoremname}{Theorem}

\begin{document}
\title[SHQP strategy for the BAP]{The supporting halfspace - quadratic programming strategy for the dual of \\the best approximation problem}

\begin{abstract}
We consider the best approximation problem (BAP) of projecting a point
onto the intersection of a number of convex sets. It is known that
Dykstra's algorithm is alternating minimization on the dual problem.
We extend Dykstra's algorithm so that it can be enhanced by the SHQP
strategy of using quadratic programming to project onto the intersection
of supporting halfspaces generated by earlier projection operations.
By looking at a structured alternating minimization problem, we show
the convergence rate of Dykstra's algorithm when reasonable conditions
are imposed to guarantee a dual minimizer. We also establish convergence
of using a warmstart iterate for Dykstra's algorithm, show how all
the results for the Dykstra's algorithm can be carried over to the
simultaneous Dykstra's algorithm, and discuss a different way of incorporating
the SHQP strategy. Lastly, we show that the dual of the best approximation
problem can have an $O(1/k^{2})$ accelerated algorithm that also
incorporates the SHQP strategy.
\end{abstract}

\author{C.H. Jeffrey Pang}

\curraddr{Department of Mathematics\\ 
National University of Singapore\\ 
Block S17 08-11\\ 
10 Lower Kent Ridge Road\\ 
Singapore 119076 }

\email{matpchj@nus.edu.sg}

\date{\today{}}


\maketitle
\tableofcontents{}

\section{Introduction}

We consider the following problem, known as the best approximation
problem (BAP).

\begin{eqnarray}
(BAP) & \min & f(x):=\frac{1}{2}\|x-d\|^{2}\label{eq:P-primal}\\
 & \mbox{s.t. } & x\in C:=C_{1}\cap\cdots\cap C_{m},\nonumber 
\end{eqnarray}
where $d$ is a given point and $C_{i}$, $i=1,\dots,m$, are closed
convex sets in a Hilbert space $X$. The BAP is equivalent to projecting
$d$ onto $C$. We shall assume throughout that $C\neq\emptyset$. 

We now recall some background on first order methods, alternating
minimization, and algorithms for the best approximation problem.

\subsection{First order methods and alternating minimization}

When presented with a problem with a large number of variables, first
order methods (which use gradient descent and avoid computationally
expensive operations like solving linear systems) and other methods
that decompose the large problems into smaller pieces to be solved
may be the only practical alternative. 

For these algorithms, the nonasymptotic or absolute rate of convergence
of the function values to the optimal objective value hold right from
the very first iteration of the algorithm, and are more useful than
the asymptotic rates. These rates are typically sublinear, like $O(1/k)$
for example. Classical references on first order methods include \cite{Nemirovsky_Yudin},
and newer references include \cite{Nesterov_book,JuditskyNemirovski_survey_a,JuditskyNemirovski_survey_b}. 

As explained in \cite{Nemirovsky_Yudin,Nesterov_book}, the nonasymptotic
rates of convergence of first order algorithms for smooth convex functions
is at best $O(1/k^{2})$. Nesterov proposed various $O(1/k^{2})$
nonasymptotic methods (which are thus optimal) for such problems (first
method \cite{Nesterov_1983}, second method \cite{Nesterov_1988_2nd_opt_method}
and third method \cite{Nesterov_smooth_min_2005}), and other optimal
methods were studied in \cite{Auslender_Teboulle_2006_SIOPT,BeckTeboulle2009,Lan_Lu_Monteiro_2011}.
The paper \cite{BeckTeboulle2009} also described an $O(1/k^{2})$
algorithm for solving the sum of a smooth convex function and a structured
nonsmooth function. These optimal methods are also known as accelerated
proximal gradient (APG) methods, and their design and analysis are
unified in the paper \cite{Tseng_APG_2008} (who also dealt with convex-concave
minimization). 

An optimization problem with a large number of variables can have
its variables divided into a number of blocks so that each subproblem
has fewer variables. These subproblems are solved in some order (often
in a cyclic manner) while the variables in other blocks are kept fixed.
See the formula \eqref{eq:CCM} for an elaboration. This is referred
to as alternating minimization (AM), and sometimes referred to Cyclic
Coordinate Minimization (CCM). Another alternative is to perform only
gradient descent on each block, which would reduce to what is described
as Cyclic Coordinate Descent, CCD. 

Methods like AM and CCD are quite old. If a function to be minimized
were nonsmooth, then it is possible for the AM and CCD to be stuck
at a non-optimal solution. A $O(1/k)$ rate of convergence for AM
of a two block problem was established in \cite{Beck_alt_min_SIOPT_2015}
without any assumption of strong convexity. As mentioned in \cite{Beck_alt_min_SIOPT_2015},
AM is also known in the literature as block-nonlinear Gauss\textendash{}Seidel
method or the block coordinate descent method (see for example \cite{Bertsekas_NLP_book}).
They also state the other contributions of \cite{Auslender_1976_book,Beck_Tetruashvili_2013,Bertsekas_NLP_book,Grippo_Sciandrone_1999,Luo_Tseng_1993_AOR}.

There has been much recent research on stochastic/ randomized CD.
Since we are not dealing with stochastic CD in this paper, we shall
only mention the papers \cite{Nesterov_2012_SIAM,Fercoq_Richtarik_2015_SIOPT},
and defer to the introduction and tables in \cite{Fercoq_Richtarik_2015_SIOPT}
for a summary of stochastic CD. A recent work \cite{Chambolle_Pock_2015_acc_BCD}
also identifies some cases where the deterministic CD scheme can have
an $O(1/k^{2})$ acceleration.

\subsection{\label{sub:BAP-MAP}The best approximation problem and the method
of alternating projections}

The BAP is often associated with the set intersection problem (SIP)
\begin{eqnarray*}
(SIP) & \mbox{Find } & x\in C:=C_{1}\cap\cdots\cap C_{m}.
\end{eqnarray*}
A well studied method for the SIP is the method of alternating projections
(MAP). We recall material from \cite{BauschkeCombettes11,Deutsch01_survey,Deustch01,EsRa11}
on material on the MAP. As its name suggests, the MAP projects the
iterates in a cyclic or some other manner so that the iterates converge
to a point in the intersection of these sets. 

One acceleration of the MAP for convex problems is the supporting
halfspace and quadratic programming strategy (SHQP): The projection
process generates supporting halfspaces of each $C_{i}$, and the
set $C$ is a subset of the polyhedron obtained by intersecting these
halfspaces. Projecting onto the polyhedron can accelerate the convergence
of the MAP, and may lead to superlinear convergence in small problems.
The SHQP strategy was discussed in \cite{cut_Pang12}. See Figure
\ref{fig:line-and-halfspace} for an illustration. This idea was discussed
in less generality in \cite{BausCombKruk06} and other papers. Other
methods of accelerating the MAP include \cite{GPR67,GK89,BDHP03}. 

As remarked by several authors, the MAP does not converge to the solution
of the BAP. Dykstra's algorithm \cite{Dykstra83} solves the best
approximation problem through a sequence of projections onto each
of the sets in a manner similar to the MAP, but correction vectors
are added before every projection. The proof of convergence to $P_{C}(d)$
was established in \cite{BD86} and sometimes referred to as the Boyle-Dykstra
theorem. Dykstra's algorithm was rediscovered by \cite{Han88}, who
showed that Dykstra's algorithm is equivalent to AM on the dual problem.
See also \cite{Gaffke_Mathar}. When the sets $C_{i}$ are halfspaces,
the convergence is asymptotically linear \cite{Deutsch_Hundal_rate_Dykstra}.
 A nonasymptotic $O(1/k)$ convergence rate of Dykstra's algorithm
was obtained in \cite{Chambolle_Pock_2015_acc_BCD} using the methods
similar to \cite{Beck_Tetruashvili_2013,Beck_alt_min_SIOPT_2015}
when a dual minimizer exists. (This does not diminish the significance
of the Boyle-Dykstra theorem. In our opinion, a quick glance at the
respective proofs shows that the Boyle-Dykstra theorem, which proves
the convergence to the primal optimal $P_{C}(d)$ even when a dual
minimizer does not exist, is technically more sophisticated than the
proof of the $O(1/k)$ convergence rate of Dykstra's algorithm when
a dual minimizer exists.) Dykstra's algorithm is quite old, so we
refer the reader to the commentary in \cite{Deustch01,EsRa11} for
more on previous work on Dykstra's algorithm.

In contrast to the situation for the MAP, not much has been done on
accelerating Dykstra's algorithm for the BAP. The question of how
to accelerate Dykstra's algorithm has been explicitly posed as an
open problem in \cite{Deustch01,Deutsch01_survey,EsRa11}. A method
was proposed in \cite{Lopez_Raydan_accel_Dykstra}. The only property
of Dykstra's algorithm needed for their acceleration is that Dykstra's
algorithm generates a sequence converging to $P_{C}(d)$, so there
is still be room for improving Dykstra's algorithm. See also \cite{Higham_Strabic_Anderson_accel}.

A variant of Dykstra's algorithm that is more suitable for parallel
computations is the simultaneous Dykstra's algorithm proposed in \cite{Iusem_DePierro_Han_1991}
using the product space formulation of \cite{Pierra84}.

Some specific best approximation problems can be solved with specialized
methods. The projection of a point into the intersection of halfspaces
can be solved by classical methods of quadratic programming. Other
sets for which the projection onto the intersection is easy include
the intersection of an affine space and the semidefinite cone \cite{Qi_Sun_SIMAX_2006,Malick_2004_SIMAX}.

The subgradient algorithm can be used to solve a convex constrained
optimization problem with a convergence rate of $O(1/\sqrt{k})$.
Hence the BAP can be solved at a rate of $O(1/\sqrt{k})$. See \cite{Nesterov_book}.
In \cite{Pang_Haugazeau}, we obtained a convergence rate of $O(1/k)$
in the case when the objective function is a strongly convex quadratic
function by adapting a Haugazeau's algorithm \cite{Haugazeau68} (see
also \cite{BauschkeCombettes11}), which is another known method for
solving the BAP. We note however that the rate of $O(1/k)$ in Haugazeau's
algorithm is typical, even when solving a BAP involving only two halfspaces.

\subsection{Contributions of this paper}

The main contribution of this paper is to extend Dykstra's algorithm
so that the SHQP strategy can be incorporated into Dykstra's algorithm.
(We try to reserve the use of the word ``acceleration'' to mean
an $O(1/k^{2})$ algorithm.) See Algorithm \ref{alg:Extended-Dykstra}
for our extension of Dykstra's algorithm. Recall that the Boyle-Dykstra
theorem proves the convergence of Dykstra's algorithm to the primal
solution of the BAP. We prove that the extended Dykstra's algorithm
also converges to the primal solution of the BAP (even when there
is no dual minimizer).

Next, we show that a commonly occurring regularity assumption guarantees
the existence of a dual minimizer. The existence of such a dual minimizer
would, by the results in \cite{Chambolle_Pock_2015_acc_BCD}, imply
that Dykstra's algorithm converges at a $O(1/k)$ rate. This analysis
also carries over to our extended Dykstra's algorithm.  

We point out that it is useful to use warmstart solutions for Dykstra's
algorithm and our extension. While it is recognized that Dykstra's
algorithm is the alternating minimization algorithm on the dual, it
appears that every description and proof of convergence of Dykstra's
algorithm in the literature starts with the default zero vector. See
further discussions in Subsection \ref{sub:Warmstart-Dykstra}. We
answer the natural question of whether Dykstra's algorithm and our
extension converge to the optimal primal solution with a warmstart
iterate by adapting the proof of the Boyle-Dykstra Theorem \cite{BD86}.
See Appendix \ref{sec:pf-extended-Dykstra}. 

We show how all these ideas mentioned earlier can be implemented for
the simultaneous Dykstra's Algorithm in Section \ref{sec:Simult-Dykstra-algorithm}.
We also explain another way to incorporate the SHQP strategy on the
BAP in Subsection \ref{sub:SHQP-Dykstra} that works when a minimizer
to the dual problem exists and is more natural to augment to the APG.
While this strategy is more natural than our extended Dyktra's algorithm,
we were not able to prove its global convergence using the framework
of the Boyle-Dykstra theorem.

\subsection{Notation}

Our notation is fairly standard. For a closed convex set $D$, we
let $P_{D}(\cdot)$ denote the projection onto $D$. The normal cone
of $D$ at a point $x\in D$ in the usual sense of convex analysis
is denoted by $N_{D}(x)$. We will let $\tilde{y}=(y_{1},\dots,y_{m})$.
When we discuss the extended Dykstra's algorithm in Section \ref{sec:Extended-Dykstra},
we will need $\tilde{y}=(y_{1},\dots,y_{m},y_{m+1})$, but this shouldn't
cause too much confusion.

\section{Preliminaries: Dykstra's algorithm }

In this section, we recall Dykstra's algorithm and some results. We
also give a discussion of warmstarting Dykstra's algorithm.
\begin{algorithm}
\label{alg:Dykstra}(Warmstart Dykstra's algorithm) Let $X$ be a
Hilbert space. Consider the problem of projecting a point $d\in X$
onto $C\subset X$, where $C=\cap_{i=1}^{m}C_{i}$ and $C_{i}$ are
closed convex sets. Choose starting $y_{i}^{(0)}\in X$ for all $i\in\{1,\dots,m\}$,
and let $x_{m}^{(0)}=d-(y_{1}^{(0)}+\cdots+y_{m}^{(0)})$.

01$\quad$For $k=1,2,\dots$ 

02$\quad$$\quad$$x_{0}^{(k)}=x_{m}^{(k-1)}$

03$\quad$$\quad$For $i=1,2,\dots,m$

04$\quad$$\quad$$\quad$$z_{i}^{(k)}:=x_{i-1}^{(k)}+y_{i}^{(k-1)}$

05$\quad$$\quad$$\quad$$x_{i}^{(k)}:=P_{C_{i}}(z_{i}^{(k)})$

06$\quad$$\quad$$\quad$$y_{i}^{(k)}:=z_{i}^{(k)}-x_{i}^{(k)}$

07$\quad$$\quad$End for 

08$\quad$End for 
\end{algorithm}
Let the vector $\tilde{y}\in X^{m}$ be $(y_{1},\dots,y_{m})$, where
each $y_{i}\in X$. For each closed convex set $D\subset X$, let
$\delta^{*}(\cdot,D):X\to\mathbb{R}$ be defined by $\delta^{*}(y,D)=\max_{x\in D}\langle y,x\rangle$.
(The function $\delta^{*}(\cdot,D)$ is also the conjugate of the
indicator function $\delta(\cdot,D)$, thus explaining our notation.)
Define the dual problem $(D')$ by 

\begin{eqnarray}
(D') & \inf_{y_{1},\dots,y_{m}} & h(y_{1},\dots,y_{m}):=f(y_{1}+\cdots+y_{m})+\sum_{i=1}^{m}\delta^{*}(y_{i},C_{i}),\label{eq:D-prime}
\end{eqnarray}
where $y_{i}\in X$ and the $f:X\to\mathbb{R}$ is as in \eqref{eq:P-primal}.

We review some easy results on $(D')$.
\begin{prop}
\label{prop:easy-Dykstra-facts}Let $X$ be a Hilbert space. Let
$C_{i}$ be closed convex sets in $X$ for $i\in\{1,\dots,m\}$, and
let $C=\cap_{i=1}^{m}C_{i}$. Let $d\in X$ and $\bar{x}=P_{C}(d)$.
Let $\tilde{y}=(y_{1},\dots,y_{m})$. We have the following:
\begin{enumerate}
\item $\inf_{y_{1},\dots,y_{m}}h(y_{1},\dots,y_{m})=\frac{1}{2}\|d\|^{2}-\frac{1}{2}\|d-\bar{x}\|^{2}.$
\item Let $v:X^{m}\to\mathbb{R}$ be defined by 
\begin{equation}
\begin{array}{c}
v(y_{1},\dots,y_{m})=\frac{1}{2}\|d-(y_{1}+\cdots+y_{m})-\bar{x}\|^{2}+\overset{m}{\underset{i=1}{\sum}}\delta^{*}(y_{i},C_{i}-\bar{x}).\end{array}\label{eq:defn-v}
\end{equation}
Then $v(\tilde{y})=h(\tilde{y})-\langle d,\bar{x}\rangle+\frac{1}{2}\|\bar{x}\|^{2}$,
and $\inf_{\tilde{y}}v(\tilde{y})=0$. 
\item We have $v(y_{1},\dots,y_{m})\geq\frac{1}{2}\|d-(y_{1}+\cdots+y_{m})-\bar{x}\|^{2}$. 
\item If $(y_{1},\dots,y_{m})$ is a minimizer of $v(\cdot)$ (or equivalently,
$h(\cdot)$), then $\bar{x}=d-(y_{1}+\cdots+y_{m})$. 
\item If $m=1$, then $y_{1}=d-\bar{x}$ is a minimizer of $v(\cdot)$ (or
equivalently, $h(\cdot)$).
\end{enumerate}
\end{prop}
\begin{proof}
Statement (1) can be obtained from \cite[pages 32--33]{Gaffke_Mathar}.
For Statement (2), note that
\[
\begin{array}{rcl}
v(\tilde{y}) & = & \frac{1}{2}\|d-(y_{1}+\cdots+y_{m})-\bar{x}\|^{2}+\overset{m}{\underset{i=1}{\sum}}\delta^{*}(y_{i},C_{i}-\bar{x})\\
 & = & \frac{1}{2}\|d-(y_{1}+\cdots+y_{m})\|^{2}-\langle d-(y_{1}+\cdots+y_{m}),\bar{x}\rangle+\frac{1}{2}\|\bar{x}\|^{2}\\
 &  & +\left[\overset{m}{\underset{i=1}{\sum}}\delta^{*}(y_{i},C_{i})\right]-\langle y_{1}+\cdots+y_{m},\bar{x}\rangle\\
 & = & \frac{1}{2}\|d-(y_{1}+\cdots+y_{m})\|^{2}-\langle d,\bar{x}\rangle+\frac{1}{2}\|\bar{x}\|^{2}+\left[\overset{m}{\underset{i=1}{\sum}}\delta^{*}(y_{i},C_{i})\right]\\
 & = & h(\tilde{y})-\langle d,\bar{x}\rangle+\frac{1}{2}\|\bar{x}\|^{2}.
\end{array}
\]
The rest of Statement (2) is elementary. Statements (3) and (4) follow
easily from the fact $0\in C_{i}-\bar{x}$, which gives $\delta^{*}(y_{i},C_{i}-\bar{x})\geq\langle y_{i},0\rangle=0$.
Statement (5) is easy.
\end{proof}
As explained in \cite{Han88,Gaffke_Mathar} and perhaps other sources,
alternating minimization in the order 

\begin{eqnarray}
y_{1}^{(k)} & = & \arg\min_{y}h(y,y_{2}^{(k-1)},y_{3}^{(k-1)},\dots,y_{m}^{(k-1)})\label{eq:CCM}\\
y_{2}^{(k)} & = & \arg\min_{y}h(y_{1}^{(k)},y,y_{3}^{(k-1)},\dots,y_{m}^{(k-1)})\nonumber \\
 & \vdots\nonumber \\
y_{m}^{(k)} & = & \arg\min_{y}h(y_{1}^{(k)},y_{2}^{(k)},\dots,y_{m-1}^{(k)},y),\nonumber 
\end{eqnarray}
leads to the Dykstra's algorithm as presented in Algorithm \ref{alg:Dykstra}
through Proposition \ref{prop:easy-Dykstra-facts}(5). We also
have the following easily verifiable facts:

\begin{eqnarray}
x_{i}^{(k)} & = & d-y_{1}^{(k)}-\cdots-y_{i-1}^{(k)}-y_{i}^{(k)}-y_{i+1}^{(k-1)}-\cdots-y_{m}^{(k-1)}\label{eq:about-x}\\
\mbox{ and }z_{i}^{(k)} & = & d-y_{1}^{(k)}-\cdots-y_{i-1}^{(k)}-y_{i+1}^{(k-1)}-\cdots-y_{m}^{(k-1)}\label{eq:about-z}
\end{eqnarray}

\subsection{\label{sub:Warmstart-Dykstra}Warmstart Dykstra's algorithm}

It appears that all descriptions and proofs of convergence of Dykstra's
algorithm use the default starting point $y_{i}^{(0)}=0$ for all
$i\in\{1,\dots,m\}$. We saw earlier that Dykstra's algorithm is alternating
minimization on the dual problem with starting point $\tilde{y}^{(0)}$.
In particular, the iterates $\tilde{y}^{(k)}$ are such that $\{h(\tilde{y}^{(k)})\}_{k}$
is a non-increasing sequence of real numbers to the dual objective
value. One may then choose a starting point $\tilde{y}^{(0)}$ such
that $h(\tilde{y}^{(0)})$ is closer to the dual objective value than
the default starting point of all zeros. There are several ways to
obtain a different starting point.
\begin{enumerate}
\item One can use greedy algorithms (that may not guarantee global convergence
to the optimal solution) to decrease the dual objective values. A
plausible strategy is to use the greedy algorithms till they do not
appear to achieve good decrease in the value $h(\cdot)$, then switch
to the warmstart Dykstra's algorithm, or our extended algorithm in
Algorithm \ref{alg:Extended-Dykstra}, to guarantee convergence to
the optimal primal solution. 
\item A warmstart solution may be available after solving a nearby problem.
For example, one might want to resolve a problem after a set has been
added or removed, or after a perturbation of parameters. Alternatively,
there may be a nearby structured problem that can be solved approximately
with less effort than the original problem. 
\end{enumerate}
The proof of convergence of Dykstra's algorithm with a different starting
point is not too different from the Boyle-Dykstra theorem. We defer
the proof to Appendix \ref{sec:pf-extended-Dykstra}, where we also
prove the convergence of our extended Dykstra's algorithm to be introduced
in Section \ref{sec:Extended-Dykstra}.

\section{\label{sec:Extended-Dykstra}Extended Dykstra's algorithm}

As mentioned in Subsection \ref{sub:BAP-MAP}, the SHQP strategy (of
collecting halfspaces containing $C$ generated by earlier projections
and then projecting onto the intersection of the halfspaces by QP)
can enhance the convergence of the method of alternating projections
for the set intersection problem. In this section, we present our
extension of Dykstra's algorithm in Algorithm \ref{alg:Extended-Dykstra}
and how it can incorporate the SHQP strategy. In order to extend the
proof of the Boyle-Dykstra theorem to establish the primal convergence
of our extended Dykstra's algorithm, we need Theorem \ref{thm:ppty-extended-Dykstra}(2).
The proof of Theorem \ref{thm:ppty-extended-Dykstra}(2) illustrates
why lines 8 and 12 of Algorithm \ref{alg:Extended-Dykstra} were designed
as such. The other parts of the Boyle-Dykstra theorem follow with
little modifications, so we defer the rest of the convergence proof
to Appendix \ref{sec:pf-extended-Dykstra}. We now present our extended
Dykstra's algorithm.
\begin{algorithm}
\label{alg:Extended-Dykstra}(Extended Dykstra's algorithm) Consider
the BAP \eqref{eq:P-primal}. Let $y_{i}^{(0)}\in X$ be the starting
dual variables for each component $i\in\{1,\dots,m\}$. We also introduce
a variable $y_{m+1}^{(k)}\in X$, with starting value $y_{m+1}^{(0)}$
being $0$, in our calculations. Let $H_{m+1}^{0}=X$. Set $x_{m+1}^{(0)}=d-\sum_{i=1}^{m+1}y_{i}^{(0)}$. 

01$\quad$For $k=1,2,\dots$ 

02$\quad$$\quad$$x_{0}^{(k)}=x_{m+1}^{(k-1)}$

03$\quad$$\quad$For $i=1,2,\dots,m$

04$\quad$$\quad$$\quad$$z_{i}^{(k)}:=x_{i-1}^{(k)}+y_{i}^{(k-1)}$

05$\quad$$\quad$$\quad$$x_{i}^{(k)}:=P_{C_{i}}(z_{i}^{(k)})$

06$\quad$$\quad$$\quad$$y_{i}^{(k)}:=z_{i}^{(k)}-x_{i}^{(k)}$

07$\quad$$\quad$End for 

08$\quad$$\quad$Let $C_{m+1}^{k}\subset X$ be such that $C\subset C_{m+1}^{k}\subset H_{m+1}^{k-1}$. 

09$\quad$$\quad$$z_{m+1}^{(k)}:=x_{m}^{(k)}+y_{m+1}^{(k-1)}$

10$\quad$$\quad$$x_{m+1}^{(k)}=P_{C_{m+1}^{k}}(z_{m+1}^{(k)})$

11$\quad$$\quad$$y_{m+1}^{(k)}=z_{m+1}^{(k)}-x_{m+1}^{(k)}$

12$\quad$$\quad$Let $H_{m+1}^{k}$ be the halfspace with normal
$y_{m+1}^{(k)}$ passing through $x_{m+1}^{(k)}$, i.e., 
\[
H_{m+1}^{k}=\{x:\langle y_{m+1}^{(k)},x-x_{m+1}^{(k)}\rangle\leq0\}.
\]

13$\quad$End for \end{algorithm}
\begin{rem}
(Designing $C_{m+1}^{k}$)  In line 8 of Algorithm \ref{alg:Extended-Dykstra},
the set $C_{m+1}^{k}$ can be chosen to be the intersection of $H_{m+1}^{k-1}$
and the halfspaces generated through earlier projections. The projection
$P_{C_{m+1}^{k}}(\cdot)$ can then be calculated easily using methods
of quadratic programming if the number of halfspaces defining $C_{m+1}^{k}$
is small. It is clear to see that Algorithm \ref{alg:Extended-Dykstra}
reduces to the original Dykstra's algorithm if we had kept $H_{m+1}^{k}=C_{m+1}^{k}=X$
for all $k\in\{1,2,\dots\}$. The choice of storing halfspaces for
$H_{m+1}^{k}$ in line 12 simplifies computations involved.
\begin{rem}
(Positioning sets of type $C_{m+1}^{k}$) If the number $m$ is large,
then one can introduce more than just one additional set of the type
$C_{m+1}^{k}$ at the end of all the original sets in an implementation
of Algorithm \ref{alg:Extended-Dykstra}. For example, one can introduce
the additional set after every fixed number of original sets so that
the quadratic programs formed will have a manageable number of halfspaces. 
\end{rem}
\end{rem}
Theorem \ref{thm:ppty-extended-Dykstra} below will be crucial in
proving that the iterates $\{x_{i}^{(k)}\}$ of Algorithm \ref{alg:Extended-Dykstra}
converges to the optimal primal solution. The proof of the Theorem
\ref{thm:ppty-extended-Dykstra} explains how the sets $C_{m+1}^{k}$
and $H_{m+1}^{k}$ were designed in order to maintain the conclusion
in Theorem \ref{thm:ppty-extended-Dykstra}(2). 
\begin{thm}
\label{thm:ppty-extended-Dykstra}(Properties of Algorithm \ref{alg:Extended-Dykstra})
In Algorithm \ref{alg:Extended-Dykstra}, define the dual function
$h^{k}:X^{m+1}\to\mathbb{R}$ at the $k$th iteration and $\bar{h}:X^{m+1}\to\mathbb{R}$
by
\begin{equation}
\begin{array}{rcl}
h^{k}(\tilde{y}) & = & \frac{1}{2}\|d-(y_{1}+\cdots+y_{m+1})\|^{2}+\bigg[\underset{i=1}{\overset{m}{\sum}}\delta^{*}(y_{i},C_{i})\bigg]+\delta^{*}(y_{m+1},H_{m+1}^{k})\\
\bar{h}(\tilde{y}) & = & \frac{1}{2}\|d-(y_{1}+\cdots+y_{m+1})\|^{2}+\bigg[\underset{i=1}{\overset{m}{\sum}}\delta^{*}(y_{i},C_{i})\bigg]+\delta^{*}(y_{m+1},C).
\end{array}\label{eq:hk}
\end{equation}
Let $\tilde{y}^{(k)}=(y_{1}^{(k)},\dots,y_{m}^{(k)},y_{m+1}^{(k)})$.
The following hold:
\begin{enumerate}
\item \textup{$h^{k-1}(\tilde{y}^{(k-1)})\geq h^{k}(\tilde{y}^{(k)})+\frac{1}{2}\sum_{i=1}^{m+1}\|y_{i}^{(k)}-y_{i}^{(k-1)}\|^{2}.$}
\item The sum $\sum_{j=1}^{\infty}\sum_{i=1}^{m+1}\|y_{i}^{(j)}-y_{i}^{(j-1)}\|^{2}$
is finite.
\end{enumerate}
\end{thm}
\begin{proof}
We have the following chain of inequalities: 
\begin{equation}
\begin{array}{cl}
 & \frac{1}{2}\|d-(y_{1}^{(k)}+\cdots+y_{m}^{(k)})-y_{m+1}^{(k-1)}\|^{2}+\delta^{*}(y_{m+1}^{(k-1)},H_{m+1}^{k-1})\\
\geq & \frac{1}{2}\|d-(y_{1}^{(k)}+\cdots+y_{m}^{(k)})-y_{m+1}^{(k-1)}\|^{2}+\delta^{*}(y_{m+1}^{(k-1)},C_{m+1}^{k})\\
\geq & \frac{1}{2}\|d-(y_{1}^{(k)}+\cdots+y_{m}^{(k)})-y_{m+1}^{(k)}\|^{2}+\delta^{*}(y_{m+1}^{(k)},C_{m+1}^{k})+\frac{1}{2}\|y_{m+1}^{(k)}-y_{m+1}^{(k-1)}\|^{2}\\
= & \frac{1}{2}\|d-(y_{1}^{(k)}+\cdots+y_{m}^{(k)})-y_{m+1}^{(k)}\|^{2}+\delta^{*}(y_{m+1}^{(k)},H_{m+1}^{k})+\frac{1}{2}\|y_{m+1}^{(k)}-y_{m+1}^{(k-1)}\|^{2}.
\end{array}\label{eq:h-ineq-chain}
\end{equation}
The first inequality comes from the fact that $C_{m+1}^{k}\subset H_{m+1}^{k-1}$,
which implies that $\delta^{*}(\cdot,H_{m+1}^{k-1})\geq\delta^{*}(\cdot,C_{m+1}^{k})$.
The second inequality comes from the fact that $y_{m+1}^{(k)}$ is
the minimizer of the strongly convex function with modulus $1$ defined
by 
\[
\begin{array}{c}
y\mapsto\frac{1}{2}\|d-(y_{1}^{(k)}+\cdots+y_{m}^{(k)})-y\|^{2}+\delta^{*}(y,C_{m+1}^{k}).\end{array}
\]
The final equation follows readily from the definition of $H_{m+1}^{k}$.
We can apply the same principle in \eqref{eq:h-ineq-chain} to show
that for all $i\in\{1,\dots,m\}$, we have 
\begin{equation}
\begin{array}{cl}
 & \frac{1}{2}\|d-y_{1}^{(k)}-\cdots y_{i-1}^{(k)}-y_{i}^{(k-1)}-y_{i+1}^{(k-1)}-\cdots-y_{m+1}^{(k-1)}\|^{2}+\delta^{*}(y_{i}^{(k-1)},C_{i})\\
\geq & \frac{1}{2}\|d-y_{1}^{(k)}-\cdots y_{i-1}^{(k)}-y_{i}^{(k)}-y_{i+1}^{(k-1)}-\cdots-y_{m+1}^{(k-1)}\|^{2}\\
 & +\delta^{*}(y_{i}^{(k)},C_{i})+\frac{1}{2}\|y_{i}^{(k)}-y_{i}^{(k-1)}\|^{2}.
\end{array}\label{eq:next-h-ineq}
\end{equation}
Combining \eqref{eq:h-ineq-chain} and \eqref{eq:next-h-ineq} gives
(1). 

From the fact that $H_{m+1}^{k}\supset C$, we have $\delta^{*}(\cdot,H_{m+1}^{k})\geq\delta^{*}(\cdot,C)$,
which in turn implies that $h^{k}(y)\geq\bar{h}(y)$. Moreover, for
each $k$, we make use of the observation in Proposition \ref{prop:easy-Dykstra-facts}(1)
to get 
\[
\inf_{y\in X^{m+1}}h^{k}(y)=\min_{y\in X^{m+1}}\bar{h}(y).
\]
Hence 
\[
\begin{array}{rcl}
\frac{1}{2}\underset{j=1}{\overset{k}{\sum}}\underset{i=1}{\overset{m+1}{\sum}}\|y_{i}^{(j)}-y_{i}^{(j-1)}\|^{2} & \leq & h^{0}(\tilde{y}^{(0)})-h^{k}(\tilde{y}^{(k)})\\
 & \leq & h^{0}(\tilde{y}^{(0)})-\bar{h}(\tilde{y}^{(k)})\\
 & \leq & h^{0}(\tilde{y}^{(0)})-\min_{y}\bar{h}(y).
\end{array}
\]
Thus (2) follows.
\end{proof}
The rest of the proof of the primal convergence of Algorithm \ref{alg:Extended-Dykstra}
is not too different from the Boyle-Dykstra theorem, so we will prove
the convergence result in Appendix \ref{sec:pf-extended-Dykstra}.

\section{\label{sec:alt-min-strategy}Convergence rate of alternating minimization
and Dykstra's algorithm}

In this section, we first recall the proof of the $O(1/k)$ convergence
rate of alternating minimization under the assumption of strong convexity
of subproblems and bounded level sets. This will then give us the
convergence rate of the function $h(\cdot)$ in the dual of Dykstra's
algorithm. We also discuss how this analysis can be carried over to
our extended Dykstra's algorithm. In Subsection \ref{sub:SHQP-Dykstra},
we introduce another more natural way to incorporate the SHQP heuristic
into Dykstra's algorithm and attains the nonasymptotic $O(1/k)$ convergence
rate when there is a dual minimizer. But we note that we are unable
to prove the global convergence to the primal optimal solution for
this new strategy.

\subsection{\label{sub:General-convergence-rate}General convergence rate result
on alternating minimization}

In this subsection, we recall that under certain conditions, alternating
minimization has a nonasymptotic convergence rate of $O(1/k)$. We
need the following result proved in \cite{Beck_Tetruashvili_2013}
and \cite{Beck_alt_min_SIOPT_2015}.
\begin{lem}
\label{lem:seq-conv-rate}(Sequence convergence rate) Let $\alpha>0$.
Suppose the sequence of nonnegative numbers $\{a_{k}\}_{k=0}^{\infty}$
is such that 
\[
a_{k}\geq a_{k+1}+\alpha a_{k+1}^{2}\mbox{ for all }k\in\{1,2,\dots\}.
\]
\end{lem}
\begin{enumerate}
\item \cite[Lemma 6.2]{Beck_Tetruashvili_2013} If furthermore, $\begin{array}{c}
a_{1}\leq\frac{1.5}{\alpha}\mbox{ and }a_{2}\leq\frac{1.5}{2\alpha}\end{array}$, then 
\[
\begin{array}{c}
a_{k}\leq\frac{1.5}{\alpha k}\mbox{ for all }k\in\{1,2,\dots\}.\end{array}
\]

\item \cite[Lemma 3.8]{Beck_alt_min_SIOPT_2015} For any $k\geq2$, 
\[
\begin{array}{c}
a_{k}\leq\max\left\{ \left(\frac{1}{2}\right)^{(k-1)/2}a_{0},\frac{4}{\alpha(k-1)}\right\} .\end{array}
\]
In addition, for any $\epsilon>0$, if 
\[
\begin{array}{c}
\begin{array}{c}
k\geq\max\left\{ \frac{2}{\ln(2)}[\ln(a_{0})+\ln(1/\epsilon)],\frac{4}{\alpha\epsilon}\right\} +1,\end{array}\end{array}
\]
then $a_{n}\leq\epsilon$. 
\end{enumerate}
The second formula refines the first by reducing the dependence of
$a_{k}$ on the first few terms of $\{a_{i}\}_{i}$. 

We now prove our general convergence rate result for alternating minimization.
The following result was discussed in \cite{Chambolle_Pock_2015_acc_BCD}
and its ideas appeared in \cite{Beck_Tetruashvili_2013,Beck_alt_min_SIOPT_2015}.

\begin{thm}
\label{thm:CCM-conv-rate}($O(1/k)$ Convergence rate of alternating
minimization) Let $f:X^{m}\to\mathbb{R}$ be a smooth convex function,
and $g_{i}:X\to\mathbb{R}$ be (not necessarily smooth) convex functions
for $i\in\{1,\dots,m\}$, Define $h:X^{m}\to\mathbb{R}$ by 
\[
\begin{array}{c}
h(y_{1},y_{2},\dots,y_{m})=f(y_{1},y_{2},\dots,y_{m})+\underset{i=1}{\overset{m}{\sum}}g_{i}(y_{i}).\end{array}
\]
 such that 
\begin{enumerate}
\item The gradient $f':X^{m}\to X^{m}$ is Lipschitz continuous with modulus
$L$, and 
\item There is a number $\mu>0$ such that for all $i\in\{1,\dots,m\}$
and fixed variables $y_{1}$, $y_{2}$, $\dots$, $y_{i-1}$, $y_{i+1}$,
$\dots$, $y_{m}$, the map 
\[
y\mapsto f(y_{1},y_{2},\dots,y_{i-1},y,y_{i+1},\dots,y_{m})
\]
 is strongly convex with modulus $\mu>0$. 
\item A minimizer $\tilde{y}^{*}=(y_{1}^{*},y_{2}^{*},\dots,y_{m}^{*})$
of $h(\cdot)$ exists. Moreover, $M_{i}$ defined by $M_{i}=\sup\{\|y_{i}^{(k)}-y_{i}^{*}\|:k\geq0\}$
is finite for all $i\in\{1,\dots,m-1\}$.
\end{enumerate}
Suppose two successive iterates $\tilde{y}^{(k-1)}=(y_{1}^{(k-1)},y_{2}^{(k-1)},\dots,y_{m}^{(k-1)})$
and $\tilde{y}^{(k)}$ defined similarly are produced by alternating
minimization described in \eqref{eq:CCM}. Let $M=\max_{i\in\{1,\dots,m-1\}}M_{i}$.
Then \textup{
\begin{equation}
\begin{array}{c}
h(\tilde{y}^{(k-1)})-h(\tilde{y}^{*})\geq h(\tilde{y}^{(k)})-h(\tilde{y}^{*})+\frac{\mu}{2(m-1)^{3}M^{2}L^{2}}[h(\tilde{y}^{(k)})-h(\tilde{y}^{*})]^{2}.\end{array}\label{eq:conv-rate-recurr}
\end{equation}
}\textup{\emph{Applying Lemma \ref{lem:seq-conv-rate} to $a_{k}:=h(\tilde{y}^{(k)})-h(\tilde{y}^{*})$
gives}} 
\[
\begin{array}{c}
h(\tilde{y}^{(k)})-h(\tilde{y}^{*})\leq\frac{1}{k}\max\{\frac{3(m-1)^{3}M^{2}L^{2}}{\mu},h(\tilde{y}^{(1)})-h(\tilde{y}^{*}),2[h(\tilde{y}^{(2)})-h(\tilde{y}^{*})]\},\end{array}
\]
and 
\[
\begin{array}{c}
h(\tilde{y}^{(k)})-h(\tilde{y}^{*})\leq\max\left\{ \left(\frac{1}{2}\right)^{(k-1)/2}[h(\tilde{y}^{(0)})-h(\tilde{y}^{*})],\frac{8(m-1)^{3}M^{2}L^{2}}{\mu(k-1)}\right\} .\end{array}
\]
\end{thm}
\begin{proof}
The proof of this result follows similar ideas as those in \cite{Chambolle_Pock_2015_acc_BCD},
which in turn appeared in \cite{Beck_Tetruashvili_2013,Beck_alt_min_SIOPT_2015}.
Since we will use elements of this proof for the proof of Theorem
\ref{thm:conv-rate-ext-Dykstra}, we now give a self contained proof.
For each $i$, let $h_{i}:X\to\mathbb{R}$ be defined by 
\[
h_{i}(y)=h(y_{1}^{(k)},y_{2}^{(k)},\dots,y_{i-1}^{(k)},y,y_{i+1}^{(k-1)},\dots,y_{m}^{(k-1)}).
\]
In other words, $h_{i}(\cdot)$ is the $i$th block of $h:X^{m}\to\mathbb{R}$.
The mapping $h_{i}(\cdot)$ has minimizer $y_{i}^{(k)}$, and is strongly
convex with modulus $\mu$ from assumption (2). Hence 
\[
\begin{array}{c}
h_{i}(y_{i}^{(k-1)})\geq h_{i}(y_{i}^{(k)})+\frac{\mu}{2}\|y_{i}^{(k)}-y_{i}^{(k-1)}\|^{2}.\end{array}
\]
Hence 
\begin{equation}
\begin{array}{c}
h(\tilde{y}^{(k-1)})-h(\tilde{y}^{*})\geq h(\tilde{y}^{(k)})-h(\tilde{y}^{*})+\underset{i=1}{\overset{m}{\sum}}\frac{\mu}{2}\|y_{i}^{(k)}-y_{i}^{(k-1)}\|^{2}.\end{array}\label{eq:rate-1st-bdd}
\end{equation}
Next, we try to find a subgradient in $\partial h(\tilde{y}^{(k)})$
by looking at the components $\partial h_{i}(\tilde{y})$. It is clear
that $0\in\partial h_{m}(y_{m}^{(k)})$. We then look at the $i$th
component of $f'(\cdot)$, which we denote by $f_{i}'(\cdot)$. For
each $i\in\{1,\dots,m\}$, the optimality conditions of each iteration
of alternating minimization implies that 
\[
0\in f'_{i}(y_{1}^{(k)},y_{2}^{(k)},\dots,y_{i-1}^{(k)},y_{i}^{(k)},y_{i+1}^{(k-1)},\dots,y_{m}^{(k-1)})+\partial g_{i}(y_{i}^{(k)}).
\]
Thus 
\[
f'_{i}(\tilde{y}^{(k)})-f'_{i}(y_{1}^{(k)},y_{2}^{(k)},\dots,y_{i-1}^{(k)},y_{i}^{(k)},y_{i+1}^{(k-1)},\dots,y_{m}^{(k-1)})\in f'_{i}(\tilde{y}^{(k)})+\partial g_{i}(y_{i}^{(k)}).
\]
Choose a subgradient $s\in\partial h(\tilde{y}^{(k)})$, with $s\in X^{m}$
such that 
\[
s_{i}=f'_{i}(\tilde{y}^{(k)})-f'_{i}(y_{1}^{(k)},y_{2}^{(k)},\dots,y_{i-1}^{(k)},y_{i}^{(k)},y_{i+1}^{(k-1)},\dots,y_{m}^{(k-1)}).
\]
We have 
\begin{equation}
\begin{array}{rcl}
\|s_{i}\| & \leq & \|f'_{i}(\tilde{y}^{(k)})-f'_{i}(y_{1}^{(k)},y_{2}^{(k)},\dots,y_{i-1}^{(k)},y_{i}^{(k)},y_{i+1}^{(k-1)},\dots,y_{m}^{(k-1)})\|\\
 & \leq & L\underset{j=i+1}{\overset{m}{\sum}}\|y_{j}^{(k)}-y_{j}^{(k-1)}\|\\
 & \leq & L\underset{j=2}{\overset{m}{\sum}}\|y_{j}^{(k)}-y_{j}^{(k-1)}\|.
\end{array}\label{eq:s-i-ineq}
\end{equation}
The above derivation also reminds us that $\|s_{m}\|=0$. Thus, making
use of condition (3), we have 
\begin{equation}
\begin{array}{rcl}
h(\tilde{y}^{*}) & \geq & h(\tilde{y}^{(k)})+\langle s,\tilde{y}^{*}-\tilde{y}^{(k)}\rangle\\
\Rightarrow h(\tilde{y}^{(k)})-h(\tilde{y}^{*}) & \leq & -\langle s,\tilde{y}^{*}-\tilde{y}^{(k)}\rangle\\
 & \leq & \underset{i=1}{\overset{m-1}{\sum}}\|s_{i}\|\|y_{i}^{*}-y_{i}^{(k)}\|\\
 & \leq & L\bigg[\underset{j=2}{\overset{m}{\sum}}\|y_{j}^{(k)}-y_{j}^{(k-1)}\|\bigg]\bigg[\underset{i=1}{\overset{m-1}{\sum}}\|y_{i}^{*}-y_{i}^{(k)}\|\bigg]\\
 & \leq & (m-1)ML\bigg[\underset{j=2}{\overset{m}{\sum}}\|y_{j}^{(k)}-y_{j}^{(k-1)}\|\bigg].
\end{array}\label{eq:rate-2nd-bdd}
\end{equation}
Applying \eqref{eq:rate-2nd-bdd} on \eqref{eq:rate-1st-bdd} gives
\begin{equation}
\begin{array}{rcl}
h(\tilde{y}^{(k-1)})-h(\tilde{y}^{*}) & \geq & h(\tilde{y}^{(k)})-h(\tilde{y}^{*})+\underset{i=1}{\overset{m}{\sum}}\frac{\mu}{2}\|y_{i}^{(k)}-y_{i}^{(k-1)}\|^{2}\\
 & \geq & h(\tilde{y}^{(k)})-h(\tilde{y}^{*})+\underset{i=2}{\overset{m}{\sum}}\frac{\mu}{2}\|y_{i}^{(k)}-y_{i}^{(k-1)}\|^{2}\\
 & \geq & h(\tilde{y}^{(k)})-h(\tilde{y}^{*})+\frac{\mu}{2(m-1)}\bigg[\underset{i=2}{\overset{m}{\sum}}\|y_{i}^{(k)}-y_{i}^{(k-1)}\|\bigg]^{2}\\
 & \geq & h(\tilde{y}^{(k)})-h(\tilde{y}^{*})+\frac{\mu}{2(m-1)^{3}M^{2}L^{2}}[h(\tilde{y}^{(k)})-h(\tilde{y}^{*})]^{2}.
\end{array}\label{eq:rate-3rd-bdd}
\end{equation}
Let $a_{k}=h(\tilde{y}^{(k)})-h(\tilde{y}^{*})$. Applying Lemma \ref{lem:seq-conv-rate}
gives us our conclusion. (For the first formula, $\alpha=\min\{\frac{\mu}{2(m-1)^{3}M^{2}L^{2}},\frac{1.5}{a_{1}},\frac{0.75}{a_{2}}\}$.)
\end{proof}
It is clear to see that condition (3) in Theorem \ref{thm:CCM-conv-rate}
is satisfied when the level sets of $h(\cdot)$ are bounded. Condition
(3) can be easily amended to having all but one of the $M_{i}$ for
$i\in\{1,\dots,m\}$ being finite.

\subsection{Convergence rate of extended Dykstra's algorithm }

In Dykstra's algorithm, the function $f(\cdot)$ in \eqref{eq:P-primal}
is quadratic, and therefore its gradient is linear. Furthermore, each
block $f_{i}(\cdot)$ is strongly convex with modulus 1. Thus conditions
(1) and (2) of Theorem \ref{thm:CCM-conv-rate} are satisfied. We
make some remarks condition (3) of Theorem \ref{thm:CCM-conv-rate}.
\begin{rem}
(Condition (3) of Theorem \ref{thm:CCM-conv-rate} for Dykstra's algorithm)
As pointed out in \cite{Han88}, there may not exist a minimizer $\tilde{y}^{*}$
of the dual problem $(D')$. Consider for example the problem of projecting
onto the intersection of two circles in $\mathbb{R}^{2}$ intersecting
at only one point. Furthermore, Gaffke and Mathar \cite[Lemma 2]{Gaffke_Mathar}
showed that for Dykstra's algorithm, if there is a $\lambda>2$ such
that $\|x_{m}^{(k)}-\bar{x}\|^{2}\in O(1/k^{\lambda})$, then $y_{i}^{*}=\lim_{k\to\infty}y_{i}^{k}$
exists with $\delta^{*}(y_{i}^{*},C_{i})$ finite, and $\tilde{y}^{*}=(y_{1}^{*},\dots,y_{m}^{*})$
minimizing the function $h(\cdot)$ of \eqref{eq:D-prime}. This result
can somewhat be seen as a converse of Theorem \ref{thm:CCM-conv-rate}.
\begin{rem}
(Finiteness of the $M_{i}$'s) In our analysis of Dykstra's algorithm,
suppose all but one of the $M_{i}$'s in Theorem \ref{thm:CCM-conv-rate}(3)
are finite for $i\in\{1,\dots,m\}$. The Boyle-Dykstra theorem implies
that the limit 
\[
\lim_{k\to\infty}[d-y_{1}^{(k)}-\cdots-y_{m}^{(k)}]=\lim_{k\to\infty}x_{m}^{(k)}
\]
 exists. This would imply that all the $M_{i}$'s are finite.
\end{rem}
\end{rem}
 We now provide the additional details to show that Algorithm \ref{alg:Extended-Dykstra}
(the extended Dykstra's algorithm) also converges at an $O(1/k)$
rate. 
\begin{thm}
\label{thm:conv-rate-ext-Dykstra}(Convergence rate of extended Dykstra's
algorithm) Consider Algorithm \ref{alg:Extended-Dykstra}. Recall
the definition of $h(\cdot)$ in \eqref{eq:D-prime}. Suppose the
following holds:
\begin{enumerate}
\item [(3$^{\prime}$)]A minimizer $\tilde{y}^{*}=(y_{1}^{*},y_{2}^{*},\dots,y_{m}^{*})$
of $h(\cdot)$ exists. Moreover, $M_{i}$ defined by $M_{i}=\sup\{\|y_{i}^{(k)}-y_{i}^{*}\|:k\geq0\}$
is finite for all $i\in\{1,\dots,m+1\}$.
\end{enumerate}
(Compare this to condition (3) of Theorem \ref{thm:CCM-conv-rate}.)
Recall the definition of $h^{k}(\cdot)$ in \eqref{eq:hk}. Then the
sequence $\{h^{k}(y_{1}^{(k)},\dots,y_{m+1}^{(k)})\}_{k},$ converges
to $h(\tilde{y}^{*})$ at a rate of $O(1/k)$. \end{thm}
\begin{proof}
We highlight the differences this proof has with that of Theorem \ref{thm:CCM-conv-rate}.
Theorem \ref{thm:ppty-extended-Dykstra}(1) shows that 
\[
\begin{array}{c}
h^{k-1}(\tilde{y}^{(k-1)})\geq h^{k}(\tilde{y}^{(k)})+\frac{1}{2}\underset{i=1}{\overset{m+1}{\sum}}\|y_{i}^{(k)}-y_{i}^{(k-1)}\|^{2},\end{array}
\]
which plays the role of \eqref{eq:rate-1st-bdd}. Next, if $\tilde{y}^{*}=(y_{1}^{*},\dots,y_{m}^{*})$
is a minimizer of $h(\cdot)$, then $(y_{1}^{*},\dots,y_{m}^{*},0)$
is a minimizer of $h^{k}(\cdot)$ for all $k$. Moreover, 
\[
h^{k}(y_{1}^{*},\dots,y_{m}^{*},0)=h(\tilde{y}^{*}).
\]
Next, we can prove an analogous result to \eqref{eq:s-i-ineq} with
$L=1$. The analogous result to \eqref{eq:rate-2nd-bdd} is 
\begin{equation}
\begin{array}{c}
h^{k}(\tilde{y}^{(k)})-h(\tilde{y}^{*})\leq mML\bigg[\underset{j=2}{\overset{m+1}{\sum}}\|y_{j}^{(k)}-y_{j}^{(k-1)}\|\bigg].\end{array}\label{eq:rate-2nd-bdd-1}
\end{equation}
The analogous result to \eqref{eq:rate-3rd-bdd} is
\begin{equation}
\begin{array}{c}
h^{k-1}(\tilde{y}^{(k-1)})-h(\tilde{y}^{*})\geq h^{k}(\tilde{y}^{(k)})-h(\tilde{y}^{*})+\frac{\mu}{2m^{3}M^{2}L^{2}}[h^{k}(\tilde{y}^{(k)})-h(\tilde{y}^{*})]^{2}.\end{array}\label{eq:rate-3rd-bdd-1}
\end{equation}
 The conclusion follows with steps similar to the proof of Theorem
\ref{thm:CCM-conv-rate}.
\end{proof}
An indicator of whether an $O(1/k)$ convergence rate is achieved
would be whether condition (3) in Theorem \ref{thm:CCM-conv-rate}
is satisfied. The next result gives sufficient conditions. 
\begin{thm}
\label{thm:condition-bdd-iters}(Condition for bounded dual iterates)
Suppose $X=\mathbb{R}^{n}$, and consider the BAP \eqref{eq:P-primal}. 
\begin{enumerate}
\item Suppose at the primal optimal solution $x^{*}=P_{C}(d)$, we have
\begin{equation}
\begin{array}{l}
\underset{i=1}{\overset{m}{\sum}}v_{i}=0\mbox{ and }v_{i}\in N_{C_{i}}(x^{*})\mbox{ for all }i\in\{1,\dots,m\}\\
\quad\mbox{ implies }v_{i}=0\mbox{ for all }i\in\{1,\dots,m\}.
\end{array}\label{eq:CQ1}
\end{equation}
Then the iterates $\{\tilde{y}^{(k)}\}$ of Dykstra's algorithm are
bounded. Moreover, an accumulation point exists, and is an optimal
solution for $(D')$, so condition (3) of Theorem \ref{thm:CCM-conv-rate}
holds.
\item Suppose at the primal optimal solution $x^{*}=P_{C}(d)$, we have
\begin{equation}
\begin{array}{l}
\underset{i=1}{\overset{m+1}{\sum}}v_{i}=0,v_{m+1}\in N_{C}(x^{*})\mbox{ and }v_{i}\in N_{C_{i}}(x^{*})\mbox{ for all }i\in\{1,\dots,m\}\\
\quad\mbox{ implies }v_{i}=0\mbox{ for all }i\in\{1,\dots,m+1\}.
\end{array}\label{eq:CQ2}
\end{equation}
Then the iterates $\{\tilde{y}^{(k)}\}$ of the extended Dykstra's
algorithm are bounded. Moreover, an accumulation point exists, and
is a minimizer of $\bar{h}:(\mathbb{R}^{n})^{m+1}\to\mathbb{R}$ defined
in \eqref{eq:hk}, so condition (3$^{\prime}$) of Theorem \ref{thm:conv-rate-ext-Dykstra}
holds.
\item Suppose $N_{C_{i}}(x^{*})$ does not contain a line for all $i\in\{1,\dots,m\}$.
In other words, the cones $N_{C_{i}}(x^{*})$ are pointed for all
$i$. Then \eqref{eq:CQ1} and \eqref{eq:CQ2} are equivalent.
\end{enumerate}
\end{thm}
\begin{proof}
For (1), we prove the boundedness of the iterates for Dykstra's algorithm.
The other parts of the result are straightforward. Seeking a contradiction,
suppose the iterates $\{\tilde{y}^{(k)}\}$ are not bounded. Then
\begin{eqnarray}
\sum_{i=1}^{m}y_{i}^{(k)} & = & d-x_{m}^{(k)}\nonumber \\
\frac{1}{\max_{i}\|y_{i}^{(k)}\|}\sum_{i=1}^{m}y_{i}^{(k)} & = & \frac{1}{\max_{i}\|y_{i}^{(k)}\|}[d-x_{m}^{(k)}].\label{eq:bdd-proof-1}
\end{eqnarray}
By the convergence of Dykstra's algorithm, $\lim_{k\to\infty}[d-x_{m}^{(k)}]$
exists. Moreover, $\limsup_{k\to\infty}\max_{i}\|y_{i}^{(k)}\|=\infty$,
so by taking a subsequence if necessary (we do not relabel), the limit
of the RHS of \eqref{eq:bdd-proof-1} is zero. Let $\hat{y}_{i}^{(k)}=\frac{y_{i}^{(k)}}{\max_{j}\|y_{j}^{(k)}\|}$.
We thus have 
\[
\begin{array}{c}
\underset{i=1}{\overset{m}{\sum}}\hat{y}_{i}^{(k)}=0.\end{array}
\]
The sequence $\{(\hat{y}_{1}^{(k)},\dots,\hat{y}_{m}^{(k)})\}_{k}$
has a convergent subsequence. Let an accumulation point be $(\hat{y}_{1}^{*},\dots,\hat{y}_{m}^{*})$.
Note that $\hat{y}_{i}^{(k)}\in N_{C_{i}}(x_{i}^{(k)})$, so $\hat{y}_{i}^{*}\in N_{C_{i}}(x^{*})$.
But not all the $\hat{y}_{i}^{*}$ are zero. This gives us the contradiction
to \eqref{eq:CQ1}.

We now show how to amend the proof of (1) to prove (2). For the extended
Dykstra's algorithm, we can obtain the formula 
\[
\begin{array}{c}
\frac{1}{\max_{i}\|y_{i}^{(k)}\|}\underset{i=1}{\overset{m+1}{\sum}}y_{i}^{(k)}=\frac{1}{\max_{i}\|y_{i}^{(k)}\|}[d-x_{m+1}^{(k)}],\end{array}
\]
which is similar to \eqref{eq:bdd-proof-1}. The sequence $\{(\hat{y}_{1}^{(k)},\dots,\hat{y}_{m+1}^{(k)})\}_{k}$
is defined similarly by $\hat{y}_{i}^{(k)}=\frac{y_{i}^{(k)}}{\max_{j}\|y_{j}^{(k)}\|}$,
and has a convergent subsequence with accumulation point $(\hat{y}_{1}^{*},\dots,\hat{y}_{m}^{*})$.
For any $c\in C$, we have 
\[
\langle\hat{y}_{m+1}^{(k)},c-x_{m+1}^{(k)}\rangle\leq0.
\]
As we take limits, we have 
\[
\langle\hat{y}_{m+1}^{*},c-x^{*}\rangle\leq0,
\]
so $\hat{y}_{m+1}^{*}\in N_{C}(x^{*})$. The same steps would imply
that \eqref{eq:CQ2} is violated, hence a contradiction.

Lastly, we prove (3). It is obvious that \eqref{eq:CQ2} implies \eqref{eq:CQ1}
(just take the particular case when $v_{m+1}=0$). We now prove that
\eqref{eq:CQ1} implies \eqref{eq:CQ2}. If \eqref{eq:CQ1} holds,
then the formula for intersection of normal cones of convex sets (see
\cite[Theorem 6.42]{RW98}) implies that 
\[
\begin{array}{c}
N_{C}(x^{*})=\underset{i=1}{\overset{m}{\sum}}N_{C_{i}}(x^{*}).\end{array}
\]
Suppose $\sum_{i=1}^{m+1}v_{i}=0$, where $v_{m+1}\in N_{C}(x^{*})$
and $v_{i}\in N_{C_{i}}(x^{*})$ for all $i\in\{1,\dots,m\}$. We
can write $v_{m+1}=\sum_{i=1}^{m+1}\tilde{v}_{i}$, where $\tilde{v}_{i}\in N_{C_{i}}(x^{*})$
for all $i\in\{1,\dots,m\}$. Then $\sum_{i=1}^{m}(v_{i}+\tilde{v}_{i})=0$,
and $(v_{i}+\tilde{v}_{i})\in N_{C_{i}}(x^{*})$. Condition \eqref{eq:CQ1}
would imply that $v_{i}+\tilde{v}_{i}=0$ for all $i\in\{1,\dots,m\}$.
Since $N_{C_{i}}(x^{*})$ contains no lines for all $i\in\{1,\dots,m\}$,
we have $v_{i}=\tilde{v}_{i}=0$ for all $i\in\{1,\dots,m\}$. This
implies that \eqref{eq:CQ2} holds.\end{proof}
\begin{rem}
We make a few remarks on Theorems \ref{thm:condition-bdd-iters} and
\ref{thm:CCM-conv-rate}.
\begin{enumerate}
\item A simple example of a line and a halfspace shows that \eqref{eq:CQ1}
and \eqref{eq:CQ2} cannot be equivalent if the conditions in (3)
were omitted. Even so, we can check that in this simple example, the
extended Dykstra's algorithm should perform better than the Dykstra's
algorithm in general, even when \eqref{eq:CQ2} fails. See Figure
\ref{fig:line-and-halfspace}. 
\item Even if condition (1) in Theorem \ref{thm:condition-bdd-iters} is
not satisfied, condition (3) of Theorem \ref{thm:CCM-conv-rate} can
hold. For example, consider the case of two (one dimensional) lines
intersecting only at the origin in $\mathbb{R}^{3}$. 
\item The condition \eqref{eq:CQ1} is well known to be equivalent to the
stability of the sets $\{C_{i}\}_{i=1}^{m}$ under perturbations.
See \cite{Kruger_06} for example. Condition \eqref{eq:CQ1} is also
important for establishing linear convergence of the method of alternating
projections for convex sets. See \cite{BB96_survey}.
\end{enumerate}
\end{rem}
\begin{figure}[!h]
\begin{tabular}{|c|c|}
\hline 
\includegraphics[scale=0.25]{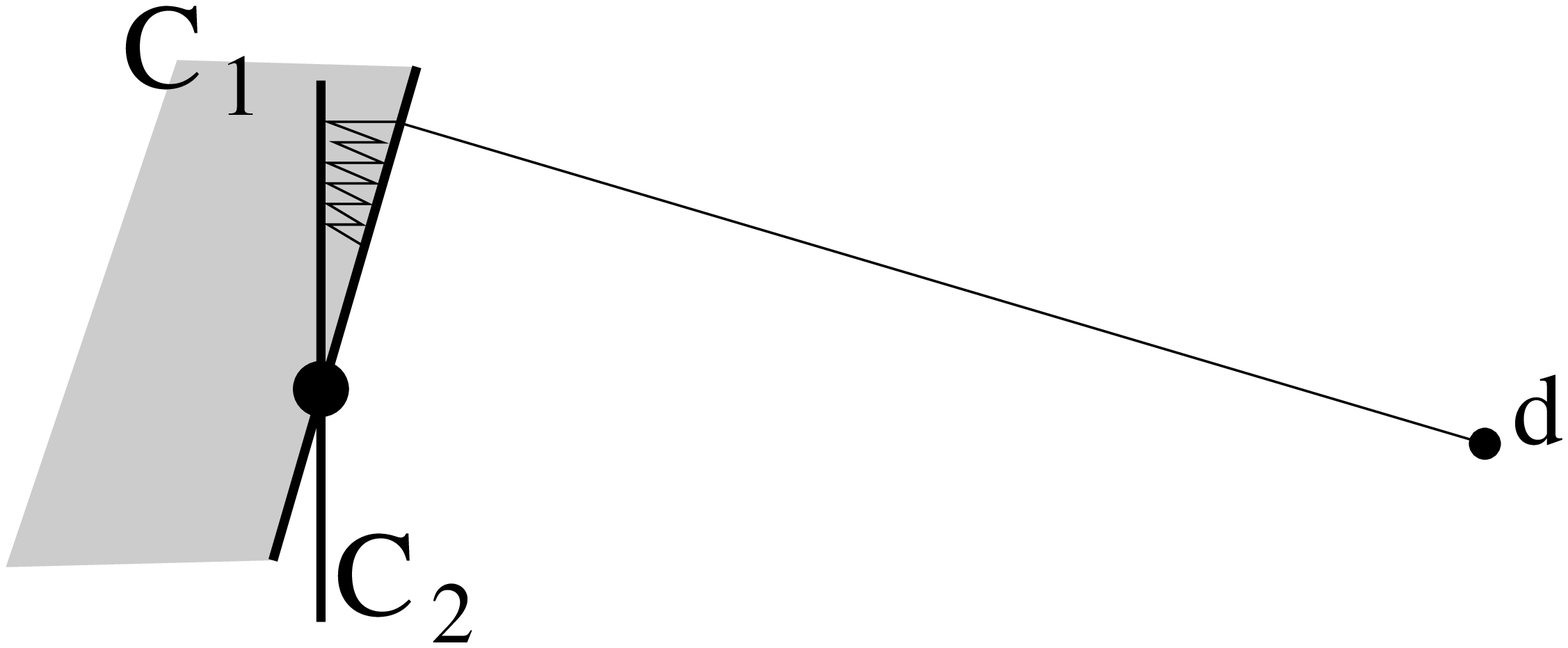} & \includegraphics[scale=0.25]{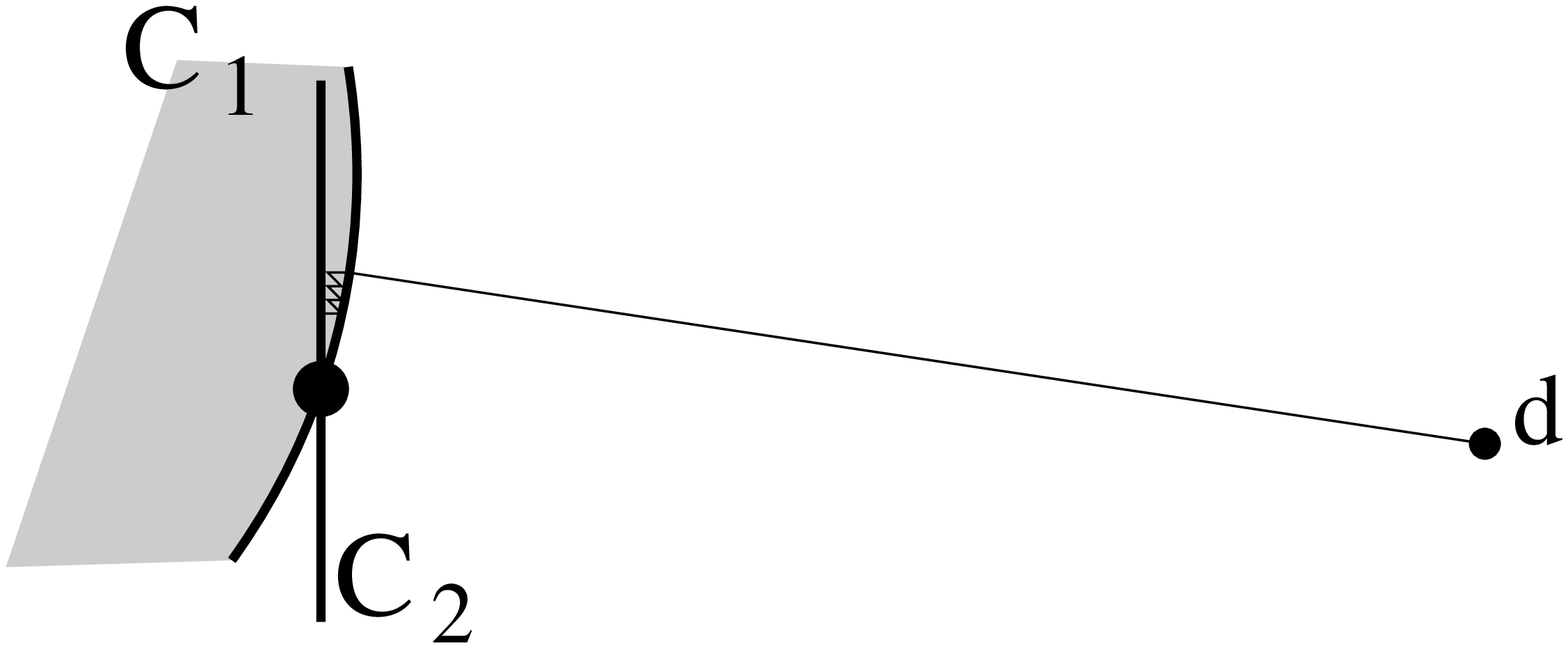}\tabularnewline
\hline 
\end{tabular}

\caption{\label{fig:line-and-halfspace}In the diagram on the left, the line
shows the path Dykstra's algorithm takes. But for both the extended
Dykstra's algorithm in Algorithm \ref{alg:Extended-Dykstra} (even
if \eqref{eq:CQ2} is not satisfied) and Algorithm \ref{alg:SHQP-orig-Dykstra},
we have convergence to $P_{C}(d)$ in a small number of steps. The
diagram on the right shows that Algorithms \ref{alg:Extended-Dykstra}
and \ref{alg:SHQP-orig-Dykstra} are also advantageous for nonpolyhedral
problems. }
\end{figure}

\subsection{\label{sub:SHQP-Dykstra}SHQP strategy for Dykstra's algorithm}

We now show that in the case where a minimizer exists for $h(\cdot)$
as defined in \eqref{eq:D-prime}, the SHQP strategy can be incorporated
into Dykstra's algorithm. We present the following additional step. 
\begin{algorithm}
\label{alg:SHQP-orig-Dykstra}(SHQP strategy for Dykstra's algorithm)
Consider the original warmstart Dykstra's algorithm (Algorithm \ref{alg:Dykstra}).
Between lines 7 and 8, we can add as many copies of the following
code segment as needed.

01 Choose $J\subset\{1,\dots,m\}$

02 Update $y_{1}^{(k)},\dots,y_{m}^{(k)}$ by solving the following
optimization problem 
\begin{eqnarray}
(y_{1}^{(k)},\dots,y_{m}^{(k)})\leftarrow & \underset{y_{1},\dots,y_{m}}{\arg\min} & \begin{array}{c}
f(y_{1}+\cdots+y_{m})+\underset{i=1}{\overset{m}{\sum}}\delta^{*}(y_{i},H_{i})\end{array}\label{eq:min-alg}\\
 & \mbox{s.t. } & \begin{array}{c}
y_{i}=y_{i}^{(k)}\mbox{ if }i\notin J.\end{array}\nonumber 
\end{eqnarray}

\end{algorithm}
To illustrate the effectiveness of the step in Algorithm \ref{alg:SHQP-orig-Dykstra},
let us for now assume that $J=\{1,\dots,m\}$. Let $y_{1}^{(k),\circ},\dots,y_{m}^{(k),\circ}\in X$
be the values of $y_{i}^{(k)}$ before line 2 was performed in Algorithm
\ref{alg:SHQP-orig-Dykstra}, and let $y_{1}^{(k),+},\dots,y_{m}^{(k),+}$
be the respective values after line 2 was performed. Note that in
Dykstra's algorithm, line 5 ($x_{i}^{(k)}=P_{C_{i}}(z_{i}^{(k)})$)
is obtained by projecting onto the set $C_{i}$, and this projection
produces a supporting halfspace $H_{i}$ at $x_{i}^{(k)}$ so that
$H_{i}\supset C_{i}$. Moreover, the halfspace $H_{i}$ also satisfies
\begin{equation}
\delta^{*}(y_{i}^{(k),\circ},C_{i})=\delta^{*}(y_{i}^{(k),\circ},H_{i}).\label{eq:eq-in-proj}
\end{equation}
 The intersection $\cap_{i=1}^{m}H_{i}$ would be a polyhedral outer
approximate of $C=\cap_{i=1}^{m}C_{i}$. Since $H_{i}\supset C_{i}$,
we have $\delta^{*}(\cdot,H_{i})\geq\delta^{*}(\cdot,C_{i})$. We
therefore have 
\begin{eqnarray*}
 &  & \begin{array}{c}
f(y_{1}^{(k),\circ}+\cdots+y_{m}^{(k),\circ})+\underset{i=1}{\overset{m}{\sum}}\delta^{*}(y_{i}^{(k),\circ},C_{i})\end{array}\\
 & \overset{\eqref{eq:eq-in-proj}}{=} & \begin{array}{c}
f(y_{1}^{(k),\circ}+\cdots+y_{m}^{(k),\circ})+\underset{i=1}{\overset{m}{\sum}}\delta^{*}(y_{i}^{(k),\circ},H_{i})\end{array}\\
 & \overset{\eqref{eq:min-alg}}{\geq} & \begin{array}{c}
f(y_{1}^{(k),+}+\cdots+y_{m}^{(k),+})+\underset{i=1}{\overset{m}{\sum}}\delta^{*}(y_{i}^{(k),+},H_{i})\end{array}\\
 & \geq & \begin{array}{c}
f(y_{1}^{(k),+}+\cdots+y_{m}^{(k),+})+\underset{i=1}{\overset{m}{\sum}}\delta^{*}(y_{i}^{(k),+},C_{i}).\end{array}
\end{eqnarray*}
Thus performing the step in line 2 of Algorithm \ref{alg:SHQP-orig-Dykstra}
improves the dual objective $h(\cdot)$. Note that the minimization
problem \eqref{eq:min-alg} is the dual of the problem of projecting
a point onto the polyhedron $\cap_{i\in J}H_{i}$, which can be solved
effectively by quadratic programming if the number of halfspaces is
small. If the number of halfspaces is large, then line 1 of Algorithm
\ref{alg:SHQP-orig-Dykstra} gives the flexibility of solving a quadratic
program of manageable size instead. In general, $H_{i}$ can be chosen
to be the intersection of halfspaces such that \eqref{eq:eq-in-proj}
is valid.

If the boundary of $C_{i}$ is smooth, then $H_{i}$ approximates
$C_{i}$ at $x_{i}^{(k)}$, and the algorithm reduces to sequential
quadratic programming. This gives a reason why the additional step
in Algorithm \ref{alg:SHQP-orig-Dykstra} can be effective in practice. 

The step explained here gives a similar kind of enhancement to what
we saw earlier for the extended Dykstra's algorithm. It is clear to
see that the recurrence \eqref{eq:conv-rate-recurr} is not affected
by the additional step in Algorithm \ref{alg:SHQP-orig-Dykstra}.
Thus the convergence analysis given in Subsection \ref{sub:General-convergence-rate}
remains valid. But when $h(\cdot)$ does not have a minimizer, we
were not able to extend the Boyle-Dykstra Theorem (specifically, Lemma
\ref{lem:9.21} below) for the proof of global convergence of the
extension of Dykstra's algorithm using Algorithm \ref{alg:SHQP-orig-Dykstra}.

\section{\label{sec:Simult-Dykstra-algorithm}Simultaneous Dykstra's algorithm}

Recall that Dykstra's algorithm reduces the best approximation problem
to a series of projections. A variant of Dykstra's algorithm which
is more suitable for parallel computations is the simultaneous Dykstra's
algorithm proposed and studied in \cite{Iusem_DePierro_Han_1991}.
In this section, we give some details on deriving the simultaneous
Dykstra's algorithm, and then show how the principles described in
extending Dykstra's algorithm can be applied for the simultaneous
Dykstra's algorithm. 

Consider the BAP \eqref{eq:P-primal}, where we want to find the projection
of $d$ onto $C=\cap_{i=1}^{m}C_{i}$. We now recall the product space
formulation of \cite{Pierra84}. Define $\mathcal{C}\subset X^{m}$
and $\mathcal{D}\subset X^{m}$ by 
\begin{eqnarray}
\mathcal{C} & := & C_{1}\times\cdots\times C_{m}\label{eq:def-C-D}\\
\mbox{ and }\mathcal{D} & := & \{(x,\dots,x)\in X^{m}:x\in X\}.\nonumber 
\end{eqnarray}
Let $\lambda_{1},\dots,\lambda_{m}$ be $m$ positive numbers that
sum to one, and let the inner product $\langle\cdot,\cdot\rangle_{\bar{Q}}$
in $X^{m}$ be defined by 
\[
\langle(u_{1},\dots,u_{m}),(v_{1},\dots v_{m})\rangle_{\bar{Q}}:=\sum_{i=1}^{m}\lambda_{i}\langle u_{i},v_{i}\rangle.
\]
The projection of the point $(d,\dots,d)\in X^{m}$ onto $\mathcal{C}\cap\mathcal{D}$
can easily be seen to be $(P_{C}(d),\dots,P_{C}(d))$. Dykstra's algorithm
can be applied onto the product space formulation. This gives the
simultaneous Dykstra's algorithm proposed and studied in \cite{Iusem_DePierro_Han_1991},
which we present below. 
\begin{algorithm}
\cite{Iusem_DePierro_Han_1991}\label{alg:Simul-Dykstra}(Simultaneous
Dykstra's algorithm) Consider the BAP \eqref{eq:P-primal}. Let $y_{i}^{(0)}\in X$
be the starting dual variables for each component $i\in\{1,\dots,m\}$.
Set $x^{(0)}=d-\sum_{i=1}^{m}\lambda_{i}y_{i}^{(0)}$. 

01$\quad$For $k=1,2,\dots$ 

02$\quad$$\quad$For $i=1,2,\dots,m$ (Parallel projection) 

03$\quad$$\quad$$\quad$$z_{i}^{(k)}:=x^{(k-1)}+y_{i}^{(k-1)}$

04$\quad$$\quad$$\quad$$x_{i}^{(k)}=P_{C_{i}}(z_{i}^{(k)})$

05$\quad$$\quad$$\quad$$y_{i}^{(k)}=z_{i}^{(k)}-x_{i}^{(k)}$

06$\quad$$\quad$end for 

07$\quad$$\quad$$x^{(k)}=\sum_{i=1}^{m}\lambda_{i}x_{i}^{(k)}$

08$\quad$end for 
\end{algorithm}
We give a brief explanation of the simultaneous Dykstra's algorithm.
Lines 3 to 5 correspond to the projection onto $\mathcal{C}$. Line
7 corresponds to projection of $(x_{1}^{(k)},\dots,x_{m}^{(k)})$
onto $\mathcal{D}$, i.e., $(x^{(k)},\dots,x^{(k)})=P_{\mathcal{D}}(x_{1}^{(k)},\dots,x_{m}^{(k)})$.
The advantage of the simultaneous Dykstra's algorithm is that lines
3 to 5 can be performed in parallel. 

We now discuss the convergence rate of the simultaneous Dykstra's
algorithm. We saw in Section \ref{sec:alt-min-strategy} that the
regularity condition \eqref{eq:CQ1} is a sufficient condition for
$O(1/k)$ convergence. We now show that this regularity condition
holds for the original problem if and only if it holds for the product
space formulation.
\begin{prop}
(Equivalence of constraint qualification) Let $C_{i}$ be closed convex
sets for $i\in\{1,\dots,m\}$, and let $C=\cap_{i=1}^{m}C_{i}$. Let
$\mathcal{C}$ and $\mathcal{D}$ be as defined in \eqref{eq:def-C-D}.
At a point $x^{*}\in C$, the conditions 

\begin{equation}
\begin{array}{l}
\underset{i=1}{\overset{m}{\sum}}v_{i}=0\mbox{ and }v_{i}\in N_{C_{i}}(x^{*})\mbox{ for all }i\in\{1,\dots,m\}\\
\quad\mbox{ implies }v_{i}=0\mbox{ for all }i\in\{1,\dots,m\}
\end{array}\label{eq:CQ1-1}
\end{equation}
and 
\begin{eqnarray}
 &  & (v_{1},\dots,v_{m})+(w_{1},\dots,w_{m})=0,\,(v_{1},\dots,v_{m})\in N_{\mathcal{C}}(x^{*},\dots,x^{*})\label{eq:CQ2-2}\\
 &  & \mbox{ and }(w_{1},\dots,w_{m})\in N_{\mathcal{D}}(x^{*},\dots,x^{*})\nonumber \\
 &  & \mbox{implies }(v_{1},\dots,v_{m})=(w_{1},\dots,w_{m})=0\nonumber 
\end{eqnarray}
 are equivalent.\end{prop}
\begin{proof}
Note that $(v_{1},\dots,v_{m})\in N_{\mathcal{C}}(x^{*},\dots,x^{*})$
if and only if $v_{i}\in N_{C_{i}}(x^{*})$ for all $i$. Next, since
$\mathcal{D}$ is a linear subspace, we have $(w_{1},\dots,w_{m})\in N_{\mathcal{D}}(x^{*},\dots,x^{*})$
if and only if $(w_{1},\dots,w_{m})\in\mathcal{D}^{\perp}$. Proposition
\ref{prop:w-in-D-perp-formula} gives the equivalent condition $\sum\lambda_{i}w_{i}=0$.
So in other words, 
\begin{eqnarray*}
 &  & (v_{1},\dots,v_{m})+(w_{1},\dots,w_{m})=0,\,(v_{1},\dots,v_{m})\in N_{\mathcal{C}}(x^{*},\dots,x^{*})\\
 &  & \mbox{ and }(w_{1},\dots,w_{m})\in N_{\mathcal{D}}(x^{*},\dots,x^{*})
\end{eqnarray*}
 is equivalent to 
\[
\begin{array}{c}
v_{i}\in N_{C_{i}}(x^{*})\mbox{ and }w_{i}=-v_{i}\mbox{ for all }i\in\{1,\dots,m\}\mbox{, and }\underset{i=1}{\overset{m}{\sum}}\lambda_{i}v_{i}=0.\end{array}
\]
Conditions \eqref{eq:CQ1-1} and \eqref{eq:CQ2-2} are now easily
seen to be equivalent. 
\end{proof}
As is well known in the study of Dykstra's algorithm, no correction
vectors for \emph{$\mathcal{D}$ }are necessary since $\mathcal{D}$
is an affine space. But we need to elaborate on the correction vector
to $\mathcal{D}$ before we show the derivation of $x^{(0)}$. Let
this correction vector be $\tilde{w}^{(k)}=(w_{1}^{(k)},\dots,w_{m}^{(k)})$.
We have $\tilde{w}^{(k)}\in N_{\mathcal{D}}(x^{(k)},\dots,x^{(k)})$.
But since $\mathcal{D}$ is a linear subspace, we have $\tilde{w}^{(k)}\in\mathcal{D}^{\perp}$.
We have the following easy result. 
\begin{prop}
\label{prop:w-in-D-perp-formula}Let $\tilde{w}=(w_{1},\dots,w_{m})$
be a vector in $X^{m}$. Then $\tilde{w}\in\mathcal{D}^{\perp}$ if
and only if $\sum\lambda_{i}w_{i}=0$.\end{prop}
\begin{proof}
This follows easily from the following chain: 
\[
\begin{array}{rl}
 & \tilde{w}\in\mathcal{D}^{\perp}\\
\iff & \langle\tilde{w},v\rangle=0\mbox{ for all }v\in\mathcal{D}\\
\iff & \langle\sum\lambda_{i}w_{i},v\rangle=0\mbox{ for all }v\in X\\
\iff & \sum\lambda_{i}w_{i}=0.
\end{array}
\]

\end{proof}
Let $\tilde{y}^{(k)}=(y_{1}^{(k)},\dots,y_{m}^{(k)})$. The default
starting vector for the simultaneous Dykstra's algorithm in \cite{Iusem_DePierro_Han_1991}
is $\tilde{y}^{(0)}=0\in X^{m}$, but we can warmstart Dykstra's algorithm
as explained in Subsection \ref{sub:Warmstart-Dykstra}. We now show
that $x^{(0)}=d-\sum\lambda_{i}y_{i}^{(0)}$ is indeed the formula
to warmstart the simultaneous Dykstra's algorithm. 
\begin{prop}
(Formula for $x^{(0)}$) In Algorithm \ref{alg:Simul-Dykstra}, for
the starting dual vector $\tilde{y}^{(k)}=(y_{1}^{(k)},\dots,y_{m}^{(k)})\in X^{m}$,
the starting iterate for $x^{(0)}$ is $x^{(0)}=d-\sum\lambda_{i}y_{i}^{(0)}$.\end{prop}
\begin{proof}
Let $\tilde{w}^{(k)}=(w_{1}^{(k)},\dots,w_{m}^{(k)})$ be the correction
vector corresponding to $\mathcal{D}$. The iterates $(x^{(k)},\dots,x^{(k)})\in X^{m}$
lie in $\mathcal{D}$ for all $k$, and $(d,\dots,d)\in\mathcal{D}$.
From our study of Dykstra's algorithm earlier, we have 
\[
(w_{1}^{(k)},\dots,w_{m}^{(k)})=(d,\dots,d)-(x^{(k)},\dots,x^{(k)})-(y_{1}^{(k)},\dots,y_{m}^{(k)}).
\]
Moreover, we have $\sum\lambda_{i}w_{i}^{(k)}=0$ from Proposition
\ref{prop:w-in-D-perp-formula}, so $\sum\lambda_{i}(d-x^{(k)}-y_{i}^{(k)})=0$.
Together with the fact that $\sum\lambda_{i}=1$, we get the needed
formula for $x^{(0)}$.
\end{proof}
We now look at how to improve Algorithm \ref{alg:Simul-Dykstra}.
Line 7 can be improved by projecting $(x_{1}^{(k)},\dots,x_{m}^{(k)})$
onto a set better than $\mathcal{D}$. Recall that line 4 produces
supporting halfspaces of the set $C_{i}$. Consider the set $C_{m+1}^{k}$
defined as the intersection of the supporting halfspaces produced
in line 4, and let $\mathcal{C}^{k}\subset X^{m}$ be defined by $\mathcal{C}^{k}=C_{m+1}^{k}\times\cdots\times C_{m+1}^{k}$
($m$ copies). We can add the set $\mathcal{C}^{k}$ to play the role
of $C_{m+1}^{k}$ in the extended Dykstra's algorithm (Algorithm \ref{alg:Extended-Dykstra})
to enhance the algorithm.

\subsection{A two-level Dykstra's algorithm}

If we want to apply the SHQP strategy to enhance the simultaneous
Dykstra's algorithm, then we might want to cut up the problem into
smaller blocks so that the quadratic programs formed are defined by
a manageable number of halfspaces. It is reasonable to assume that
information about the sets $C_{i}$ communicate upwards from the leaves
to the root of a tree (in the sense of graph theory). We illustrate
with an example with $m=4$ where we break down the size of the quadratic
programs to be at most 2. Let the sets $\mathcal{D}_{1}$ and $\mathcal{D}_{2}$
be defined by 
\begin{eqnarray*}
\mathcal{D}_{1} & = & \{(x_{1},x_{2},x_{3},x_{4})\in X^{4}:x_{1}=x_{2}\}\\
\mbox{ and }\mathcal{D}_{2} & = & \{(x_{1},x_{2},x_{3},x_{4})\in X^{4}:x_{3}=x_{4}\}.
\end{eqnarray*}
We present a two level Dykstra's algorithm. 
\begin{algorithm}
\label{thm:Two-level-Dykstra's}(Two level Dykstra's algorithm) Consider
the BAP \eqref{eq:P-primal} where $m=4$. Let $y_{i}^{(0)}\in X$
be the starting dual variables for each component $i\in\{1,\dots,4\}$.
Set $x^{(0)}=d-\sum_{i=1}^{4}\lambda_{i}y_{i}^{(0)}$. 

01$\quad$For $k=1,2,\dots$

02$\quad$$\quad$For $i\in\{1,2,3,4\}$ 

03$\quad$$\quad$$\quad$$z_{i}^{(k)}:=x^{(k-1)}+y_{i}^{(k-1)}$

04$\quad$$\quad$$\quad$$x_{i}^{(k)}=P_{C_{i}}(z_{i}^{(k)})$

05$\quad$$\quad$$\quad$$y_{i}^{(k)}=z_{i}^{(k)}-x_{i}^{(k)}$

06$\quad$$\quad$end for

07$\quad$$\quad$$x_{(1,2)}^{(k)}=\frac{\lambda_{1}}{\lambda_{1}+\lambda_{2}}x_{1}^{(k)}+\frac{\lambda_{2}}{\lambda_{1}+\lambda_{2}}x_{2}^{(k)}$

08$\quad$$\quad$\textup{$x_{(3,4)}^{(k)}=\frac{\lambda_{3}}{\lambda_{3}+\lambda_{4}}x_{3}^{(k)}+\frac{\lambda_{4}}{\lambda_{3}+\lambda_{4}}x_{4}^{(k)}$}

09$\quad$$\quad$\textup{$x^{(k)}=(\lambda_{1}+\lambda_{2})x_{(1,2)}^{(k)}+(\lambda_{3}+\lambda_{4})x_{(3,4)}^{(k)}$}

10$\quad$end for
\end{algorithm}
Lines 2 to 6 describe the operation involved in projecting onto $\mathcal{C}$,
which is not different from the simultaneous Dykstra's algorithm (Algorithm
\ref{alg:Simul-Dykstra}). Line 7 describes the operation in projecting
onto $\mathcal{D}_{1}$, line 8 describes the operation in projecting
onto $\mathcal{D}_{2}$, and line 9 describes the operation in projecting
onto $\mathcal{D}$. 

The agent that collects information on $x_{1}^{(k)}$ and $x_{2}^{(k)}$
to obtain $x_{(1,2)}^{(k)}$ can also collect the halfspaces generated
by the projection operation used to obtain $x_{1}^{(k)}$ and $x_{2}^{(k)}$.
We can make use of these halfspaces to form a superset of $C$ that
plays the role of $C_{m+1}^{k}$ in the extended Dykstra's algorithm
(Algorithm \ref{alg:Extended-Dykstra}). In other words, the operations
in lines 8 to 12 of Algorithm \ref{alg:Extended-Dykstra} can be inserted
between lines 7 and 8 of Algorithm \ref{thm:Two-level-Dykstra's}.
We can also insert these same lines between lines 8 and 9 and between
lines 9 and 10 to enhance Algorithm \ref{thm:Two-level-Dykstra's}.
It is now easy to extend the principles highlighted here for problems
involving $m>4$ sets and with more than 2 levels.

\section{\label{sec:greedy-app}Using the APG for $(D')$}

In this section, we depart from the dual alternating minimization
strategy treated in the rest of the paper, and discuss using the accelerated
proximal gradient (APG) algorithm to solve $(D')$ in \eqref{eq:D-prime}
in order to get a $O(1/k^{2})$ convergence rate. We remark that the
APG can be augmented by the strategy described in Subsection \ref{sub:SHQP-Dykstra}.

We recall the APG as presented in \cite[Section 3]{Tseng_APG_2008},
which traces its roots to Nesterov's second optimal method \cite{Nesterov_1988_2nd_opt_method}.
We decide that it is best to adopt the notation of \cite{Tseng_APG_2008}
even though it conflicts with some of the notation we have used in
the rest of the paper. 
\begin{algorithm}
\label{alg:Tseng-alg-1}\cite[Algorithm 1]{Tseng_APG_2008} Consider
the problem of minimizing 
\[
h(x)=f(x)+P(x),
\]
where $f:X\to\mathbb{R}$ is a smooth convex function whose gradient
$\nabla f:X\to X$ is Lipschitz with constant $L$, and $P:X\to\mathbb{R}$
is a (not necessarily smooth) convex function. For each $y\in X$,
define $l_{f}(\cdot;y):X\to\mathbb{R}$ (a linearization of $h(\cdot)$
at $y$) by 
\[
l_{f}(x;y)=f(y)+\langle\nabla f(y),x-y\rangle+P(x).
\]
Choose $\theta_{0}\in(0,1]$, $x_{0}$, $z_{0}\in\dom(P)$. $k\leftarrow0$.
Go to 1.
\begin{enumerate}
\item Choose a nonempty closed convex set $X_{k}\subset X$ with $X_{k}\cap\dom(P)\neq\emptyset$.
Let 
\begin{eqnarray}
y_{k} & = & \begin{array}{c}
(1-\theta_{k})x_{k}+\theta_{k}z_{k}\end{array},\label{eq:Tseng-11}\\
z_{k+1} & = & \begin{array}{c}
\arg\min_{x\in X_{k}}\{l_{f}(x;y_{k})+\frac{\theta_{k}L}{2}\|x-z_{k}\|^{2}\},\end{array}\label{eq:Tseng-12}\\
\hat{x}_{k+1} & = & \begin{array}{c}
(1-\theta_{k})x_{k}+\theta_{k}z_{k+1}.\end{array}\label{eq:Tseng-13}
\end{eqnarray}
Choose $x_{k+1}$ such that 
\begin{equation}
\begin{array}{c}
h(x_{k+1})\leq l_{f}(\hat{x}_{k+1};y_{k})+\frac{L}{2}\|\hat{x}_{k+1}-y_{k}\|^{2}.\end{array}\label{eq:Tseng-14}
\end{equation}
Choose $\theta_{k+1}\in(0,1]$ satisfying 
\begin{equation}
\begin{array}{c}
\frac{1-\theta_{k+1}}{\theta_{k+1}^{2}}\leq\frac{1}{\theta_{k}^{2}}.\end{array}\label{eq:Tseng-15}
\end{equation}
$k\leftarrow k+1$, and go to 1. 
\end{enumerate}
\end{algorithm}
The following is the convergence result of Algorithm \ref{alg:Tseng-alg-1}.
We simplify their result by taking $X_{k}=X$ for all $k$. 
\begin{thm}
\label{thm:Tseng-thm-1a}\cite[Corollary 1(a)]{Tseng_APG_2008} Let
$\{(x_{k},y_{k},z_{k},\theta_{k},X_{k})\}$ be generated by Algorithm
\ref{alg:Tseng-alg-1} with $\theta_{0}=1$. Fix any $\epsilon>0$.
Suppose $\theta_{k}\leq\frac{2}{k+2}$ (which is the case when $\theta_{0}=1$
and $\theta_{k+1}$ is determined from $\theta_{k}$ by setting \eqref{eq:Tseng-15}
to an equation), and $X_{k}=X$ for all $k$. Then for any $x\in\dom(P)$
with $h(x)\leq\inf(h)+\epsilon$, we have 
\[
\min_{i=0,1,\dots,k+1}\{h(x_{i})\}\leq h(x)+\epsilon\mbox{ whenever }\begin{array}{c}
k\geq\sqrt{\frac{4L}{\epsilon}}\|x-z_{0}\|-2.\end{array}
\]

\end{thm}
Even though the line \eqref{eq:Tseng-14} is different from that in
\cite[(14)]{Tseng_APG_2008}, it is easy to check that the inequality
\cite[(23)]{Tseng_APG_2008} remains valid with this change.

Theorem \ref{thm:Tseng-thm-1a} shows that the infimum of $\{h(x_{k})\}_{k}$
produced by Algorithm \ref{alg:Tseng-alg-1} would converge to the
infimum of $h(\cdot)$. Furthermore, if a minimizer of $h(\cdot)$
exists, then the convergence rate of $\{h(x_{k})-\inf(h)\}_{k}$ is
of $O(1/k^{2})$. 

For the BAP \eqref{eq:P-primal} of projecting a point onto the intersection
of $m$ sets, the function $h(\cdot)$ was described in \eqref{eq:D-prime}.
For $\tilde{y}\in X^{m}$, the mapping 
\begin{equation}
\tilde{y}\mapsto\|d-\sum y_{i}\|^{2}=f(y_{1}+\cdots+y_{m})\label{eq:f-map}
\end{equation}
 has Hessian 
\[
\left(\begin{array}{cccc}
I & I & \cdots & I\\
I & I &  & I\\
\vdots &  & \ddots & \vdots\\
I & I & \cdots & I
\end{array}\right)
\]
(i.e., there are $m^{2}$ blocks in an $m\times m$ block square matrix),
and the gradient of the map in \eqref{eq:f-map} is Lipschitz with
constant $L=m$. The step \eqref{eq:Tseng-12} can now be easily carried
out using Proposition \ref{prop:easy-Dykstra-facts}(5) to obtain
all $m$ components of the minimizer $z_{i+1}$. We can use the strategy
described in Subsection \ref{sub:SHQP-Dykstra} to get a better iterate
$x_{k+1}$ satisfying \eqref{eq:Tseng-14} than $\hat{x}_{k+1}$.

\section{Conclusion}

In this paper, we showed ways to incorporate the SHQP heuristic to
improve Dykstra's algorithm. For the case when $C_{i}$ are hyperplanes,
the numerical experiments in \cite{Pang_subBAP} shows the effectiveness
of the strategies explained in this paper. We defer further numerical
experiments to future work.

\appendix

\section{\label{sec:pf-extended-Dykstra}Proof of convergence of Algorithm
\ref{alg:Extended-Dykstra}}

In this appendix, we present the proof of convergence of Algorithm
\ref{alg:Extended-Dykstra}, the extended Dykstra's algorithm. We
already saw that if $H_{m+1}^{k}=C_{m+1}^{k}=X$ for all $k\geq0$,
then Algorithm \ref{alg:Extended-Dykstra} reduces to the original
Dykstra's algorithm. Apart from Theorem \ref{thm:ppty-extended-Dykstra},
our proof is mostly the same as the Boyle-Dykstra theorem \cite{BD86}
as presented in \cite{Deustch01}. Note that the proof here also includes
the warmstart case. 

Throughout this section, we follow the notation of Algorithm \ref{alg:Extended-Dykstra}.
We need to follow the notation in \cite{Deustch01} and define the
sequences $\{e_{i}\}_{i=-m}^{\infty}$ and $\{\tilde{x}_{i}\}_{i=0}^{\infty}$
by 
\begin{eqnarray}
e_{(m+1)(k-1)+i} & = & y_{i}^{(k)}\label{eq:def-e}\\
\tilde{x}_{(m+1)(k-1)+i} & = & x_{i}^{(k)}.\label{eq:def-x}
\end{eqnarray}
The statement of Lemma \ref{lem:9-23} makes the new notation more
natural. We denote $[i]$ to be the integer in $\{1,\dots,m+1\}$
such that $m+1$ divides $i-[i]$. 
\begin{lem}
\label{lem:9.17}In Algorithm \ref{alg:Extended-Dykstra}, for each
$i\geq1$, such that $[i]\in\{1,\dots,m\}$. 
\begin{equation}
\delta^{*}(e_{i},C_{[i]}-y)=\left\langle \tilde{x}_{i}-y,e_{i}\right\rangle \geq0\mbox{ for all }y\in C_{[i]}.\label{eq:9.17.1}
\end{equation}
Furthermore, if $[i]=m+1$, then 
\begin{equation}
\delta^{*}(e_{i},C_{m+1}^{i/(m+1)}-y)=\left\langle \tilde{x}_{i}-y,e_{i}\right\rangle \geq0\mbox{ for all }y\in C.\label{eq:9.17.2}
\end{equation}
\end{lem}
\begin{proof}
The proof of inequality \eqref{eq:9.17.1} is exactly the same as
\cite[Lemma 9.17]{Deustch01}, but our statement is now only valid
for all $n\geq1$. We have 
\begin{eqnarray*}
 &  & \left\langle \tilde{x}_{i}-y,e_{i}\right\rangle \\
 & = & \left\langle P_{C_{[i]}}(\tilde{x}_{i-1}+e_{i-(m+1)})-y,\tilde{x}_{i-1}+e_{i-(m+1)}-P_{K_{[i]}}(\tilde{x}_{i-1}+e_{i-(m+1)})\right\rangle \geq0,
\end{eqnarray*}
where the inequality is an immediate consequence from the properties
of projections. The second inequality in \eqref{eq:9.17.2} is also
clear. The equations in both \eqref{eq:9.17.1} and \eqref{eq:9.17.2}
are straightforward from the definition of $\delta^{*}(\cdot,\cdot)$.\end{proof}
\begin{lem}
\label{lem:9-18}In Algorithm \ref{alg:Extended-Dykstra}, for each
$i\geq0$,
\begin{equation}
d-\tilde{x}_{i}=e_{i-m}+e_{i-(m-1)}+\cdots+e_{i-1}+e_{i}.\label{eq:9.18.1}
\end{equation}
\end{lem}
\begin{proof}
This is easily seen from lines 4 and 9 of Algorithm \ref{alg:Extended-Dykstra}
and the formula for $z_{i}^{(k)}$ in \eqref{eq:about-z}. \end{proof}
\begin{lem}
\label{lem:9.20}In Algorithm \ref{alg:Extended-Dykstra}, $\{\tilde{x}_{i}\}$
is a bounded sequence, and 
\begin{equation}
\begin{array}{c}
\underset{i=1}{\overset{\infty}{\sum}}\|\tilde{x}_{i-1}-\tilde{x}_{i}\|^{2}<\infty.\end{array}\label{eq:9.20.1}
\end{equation}
In particular, 
\begin{equation}
\|\tilde{x}_{i-1}-\tilde{x}_{i}\|\to0\mbox{ as }i\to\infty.\label{eq:9.20.2}
\end{equation}
\end{lem}
\begin{proof}
Formula \eqref{eq:9.20.1} is just a rephrasing of Theorem \ref{thm:ppty-extended-Dykstra}(2).
Formula \eqref{eq:9.20.2} follows easily. 

We now show the boundedness of $\{\tilde{x}_{i}\}$. For $i$, let
$k=\lfloor\frac{i}{m+1}\rfloor$. Define $v_{i}$ as 
\begin{equation}
v_{i}:=\frac{1}{2}\|\tilde{x}_{i}-P_{C}(d)\|^{2}+\sum_{l=i-m}^{i}\langle e_{l},\tilde{x}_{l}-P_{C}(d)\rangle.\label{eq:def-v-i}
\end{equation}
Recall the definition of $h^{k}(\cdot)$ in \eqref{eq:hk}. We have
\begin{eqnarray*}
v_{i} & = & \frac{1}{2}\|\tilde{x}_{i}-P_{C}(d)\|^{2}+\sum_{l=i-m}^{i}\langle e_{l},\tilde{x}_{l}-P_{C}(d)\rangle\\
 & = & \frac{1}{2}\|x_{i-k(m+1)}^{(k)}-P_{C}(d)\|^{2}+\sum_{l=1}^{i-k(m+1)}\delta^{*}(y_{l}^{(k)},C_{[l]}-P_{C}(d))\\
 &  & +\delta^{*}(y_{m+1}^{(k-1)},C_{m+1}^{k-1}-P_{C}(d))+\sum_{l=i-k(m+1)+1}^{m}\delta^{*}(y_{l}^{(k-1)},C_{[l]}-P_{C}(d))\\
 & = & h^{k-1}(y_{1}^{(k)},y_{2}^{(k)},\dots,y_{i-k(m+1)}^{(k)},y_{i-k(m+1)+1}^{(k-1)},\dots,y_{m+1}^{(k-1)})\\
 &  & -\langle d,P_{C}(d)\rangle+\frac{1}{2}\|P_{C}(d)\|^{2}.
\end{eqnarray*}
The proof of Theorem \ref{thm:ppty-extended-Dykstra} shows that $v_{i}$
is non-increasing.Since $0\in C_{m+1}^{k-1}-P_{C}(d)$ and $0\in C_{[l]}-P_{C}(d)$,
we have $v_{i}\geq\frac{1}{2}\|\tilde{x}_{i}-P_{C}(d)\|^{2}$ (just
like in Proposition \ref{prop:easy-Dykstra-facts}(3)), which shows
that $\{\tilde{x}_{i}\}$ is a bounded sequence.\end{proof}
\begin{lem}
\label{lem:9.21}In Algorithm \ref{alg:Extended-Dykstra}, for any
$i\in\mathbb{N}$, 
\begin{equation}
\|e_{i}\|\leq\sum_{k=1}^{i}\|\tilde{x}_{k-1}-\tilde{x}_{k}\|+\max_{1\leq l\leq m+1}\|e_{l-(m+1)}\|.\label{eq:9.21.1}
\end{equation}
\end{lem}
\begin{proof}
The proof is adjusted from \cite[Lemma 9.21]{Deustch01}. We induct
on $i$. It is clear to see that \eqref{eq:9.21.1} holds for all
$i\in\{-m,\dots,0\}$. Suppose \eqref{eq:9.21.1} holds for all $r\leq i$.
Let $M_{1}=\max_{1\leq l\leq m+1}\|e_{l-(m+1)}\|$. Then 
\[
\begin{array}{rcl}
\|e_{i+1}\| & = & \|\tilde{x}_{i}-\tilde{x}_{i+1}+e_{i+1-(m+1)}\|\leq\|\tilde{x}_{i}-\tilde{x}_{i+1}\|+\|e_{i+1-(m+1)}\|\\
 & \leq & \|\tilde{x}_{i}-\tilde{x}_{i+1}\|+\underset{k=1}{\overset{i+1-(m+1)}{\sum}}\|\tilde{x}_{k-1}-\tilde{x}_{k}\|+M_{1}\leq\underset{k=1}{\overset{i+1}{\sum}}\|\tilde{x}_{k-1}-\tilde{x}_{k}\|+M_{1},
\end{array}
\]
which implies that \eqref{eq:9.21.1} holds for $r=i+1$.\end{proof}
\begin{lem}
\label{lem:9-22}In Algorithm \ref{alg:Extended-Dykstra}, 
\begin{equation}
\begin{array}{c}
\liminf_{i}\underset{k=i-m}{\overset{i}{\sum}}|\left\langle \tilde{x}_{k}-\tilde{x}_{i},e_{k}\right\rangle |=0.\end{array}\label{eq:9.22.1}
\end{equation}
\end{lem}
\begin{proof}
The proof needs to be adjusted from \cite[Lemma 9.22]{Deustch01}.
Let $M_{1}=\max_{1\leq l\leq m+1}\|e_{l-(m+1)}\|$. Using Schwarz's
inequality and Lemma \ref{lem:9.21}, we get 
\[
\begin{array}{cl}
 & \underset{k=i-m}{\overset{i}{\sum}}|\left\langle \tilde{x}_{k}-\tilde{x}_{i},e_{k}\right\rangle |\\
\leq & \underset{k=i-m}{\overset{i}{\sum}}\|e_{k}\|\|\tilde{x}_{k}-\tilde{x}_{i}\|\\
\leq & \underset{k=i-m}{\overset{i}{\sum}}\left[\left(M_{1}+\underset{j=1}{\overset{k}{\sum}}\|\tilde{x}_{j-1}-\tilde{x}_{j}\|\right)\|\tilde{x}_{k}-\tilde{x}_{i}\|\right]\\
\leq & \underset{k=i-m}{\overset{i}{\sum}}\left[\left(\underset{j=1}{\overset{k}{\sum}}\|\tilde{x}_{j-1}-\tilde{x}_{j}\|\right)\left(\underset{l=k+1}{\overset{i}{\sum}}\|\tilde{x}_{l-1}-\tilde{x}_{l}\|\right)\right]+M_{1}\underset{k=i-m}{\overset{i}{\sum}}\|\tilde{x}_{k}-\tilde{x}_{i}\|\\
\leq & \underbrace{(m+1)\left(\underset{j=1}{\overset{i}{\sum}}\|\tilde{x}_{j-1}-\tilde{x}_{j}\|\right)\left(\underset{l=i-(m-1)}{\overset{i}{\sum}}\|\tilde{x}_{l-1}-\tilde{x}_{l}\|\right)}_{(1)}+\underbrace{M_{1}\underset{k=i-m}{\overset{i}{\sum}}\|\tilde{x}_{k}-\tilde{x}_{i}\|}_{(2)}.
\end{array}
\]
Term $(2)$ converges to zero by Lemma \ref{lem:9.20}. Let $a_{i}=\|x_{i-1}-x_{i}\|$.
To show our result, it suffices to show that 
\[
\begin{array}{c}
\liminf_{i}\left[\left(\underset{j=1}{\overset{i}{\sum}}a_{j}\right)\left(\underset{l=i-(m-1)}{\overset{i}{\sum}}a_{l}\right)\right]=0\end{array}
\]
 given that $\sum_{j=1}^{\infty}a_{j}^{2}$ is finite. We refer the
reader to the proof in \cite[Lemma 9.22]{Deustch01} for the proof
of this fact.\end{proof}
\begin{lem}
\label{lem:9-23}In Algorithm \ref{alg:Extended-Dykstra}, there exists
a subsequence $\{\tilde{x}_{i_{j}}\}$ of $\{\tilde{x}_{i}\}$ such
that 
\begin{equation}
\limsup_{j}\langle y-\tilde{x}_{i_{j}},d-\tilde{x}_{i_{j}}\rangle\leq0\mbox{ for each }y\in C,\mbox{ and }\label{eq:9.23.1}
\end{equation}
\begin{equation}
\begin{array}{c}
\lim_{j}\underset{k=i_{j}-m}{\overset{i_{j}}{\sum}}|\langle\tilde{x}_{k}-\tilde{x}_{i_{j}},e_{k}\rangle|=0.\end{array}\label{eq:9.23.2}
\end{equation}
\end{lem}
\begin{proof}
The proof is almost exactly the same as \cite[Lemma 9.23]{Deustch01}.
Using Lemma \ref{lem:9-18}, we have for all $y\in C$, $i\geq m$
that 
\[
\begin{array}{rcl}
\langle y-\tilde{x}_{i},d-\tilde{x}_{i}\rangle & = & \langle y-\tilde{x}_{i},e_{i-m}+e_{i-m+1}+\cdots+e_{i}\rangle\\
 & = & \underset{k=i-m}{\overset{i}{\sum}}\langle y-\tilde{x}_{i},e_{k}\rangle\\
 & = & \underset{k=i-m}{\overset{i}{\sum}}\langle y-\tilde{x}_{k},e_{k}\rangle+\underset{k=i-m}{\overset{i}{\sum}}\langle\tilde{x}_{k}-\tilde{x}_{i},e_{k}\rangle.
\end{array}
\]
By Lemma \ref{lem:9.17}, the first sum is no more than $0$. Hence
\begin{equation}
\begin{array}{c}
\langle y-\tilde{x}_{i},d-\tilde{x}_{i}\rangle\leq\underset{k=i-m}{\overset{i}{\sum}}\langle\tilde{x}_{k}-\tilde{x}_{i},e_{k}\rangle.\end{array}\label{eq:9-23-4}
\end{equation}
By Lemma \ref{lem:9-22}, we deduce that there is a subsequence $\{i_{j}\}_{j}$
such that \eqref{eq:9.23.2} holds. Note that the right hand side
of \eqref{eq:9-23-4} does not depend on $y$. In view of \eqref{eq:9-23-4},
it follows that \eqref{eq:9.23.1} also holds.\end{proof}
\begin{thm}
\label{thm:warmstart-Boyle-Dykstra}(Warmstart Boyle-Dykstra Theorem)
Consider Algorithm \ref{alg:Extended-Dykstra}. Define the sequence
$\{\tilde{x}_{n}\}$ as in Step 2 of Algorithm \ref{alg:Extended-Dykstra}
and \eqref{eq:def-x}. Then 
\[
\lim_{i}\|\tilde{x}_{i}-P_{C}(d)\|=0.
\]
\end{thm}
\begin{proof}
The proof of this result is mostly the same as \cite[Lemma 9.23]{Deustch01}.
By Lemma \ref{lem:9-23}, there exists a subsequence $\{\tilde{x}_{i_{j}}\}$
such that 
\begin{equation}
\limsup_{j}\langle y-\tilde{x}_{i_{j}},d-\tilde{x}_{i_{j}}\rangle\leq0\mbox{ for each }\ensuremath{y\in C}.\label{eq:9.24.2}
\end{equation}
Since $\{\tilde{x}_{i}\}$ is bounded by Lemma \ref{lem:9.20}, it
follows by \cite[Theorem 9.12]{Deustch01} (by passing to a further
subsequence if necessary), that there is a $y_{0}\in X$ such that
\begin{equation}
\tilde{x}_{i_{j}}\xrightarrow{w}y_{0},\label{eq:9.24.3}
\end{equation}
and 
\begin{equation}
\lim_{j}\|\tilde{x}_{i_{j}}\|\mbox{ exists.}\label{eq:9.24.4}
\end{equation}
By another property of Hilbert spaces (\cite[Theorem 9.13]{Deustch01}),
\begin{equation}
\|y_{0}\|\leq\liminf_{j}\|\tilde{x}_{i_{j}}\|=\lim_{j}\|\tilde{x}_{i_{j}}\|.\label{eq:9.24.5}
\end{equation}
Since $[i]$ takes on only $m+1$ possibilities, an infinite number
of the $i_{j}$'s must be of the same value. If this value is in $\{1,\dots,m\}$,
say $i_{0}$, then an infinite number of the $\tilde{x}_{i_{j}}$'s
lie in $C_{i_{0}}$. Since $C_{i_{0}}$ is closed and convex, it is
weakly closed by \cite[Theorem 9.16]{Deustch01}, and hence $y_{0}\in C_{i_{0}}$.
By \eqref{eq:9.20.2}, $\tilde{x}_{i}-\tilde{x}_{i-1}\to0$. By a
repeated application of this fact, we see that all the sequences $\{\tilde{x}_{i_{j}+1}\}$,
$\{\tilde{x}_{i_{j}+2}\}$, $\dots$ converge weakly to $y_{0}$,
and hence $y_{0}\in C_{j}$ for every $j$. That is, 
\[
y_{0}\in C.
\]
For any $y\in C$, \eqref{eq:9.24.5} and \eqref{eq:9.24.2} imply
that 
\begin{eqnarray}
\langle y-y_{0},d-y_{0}\rangle & = & \langle y,d\rangle-\langle y,y_{0}\rangle-\langle y_{0},d\rangle+\|y_{0}\|^{2}\label{eq:9.24.6p}\\
 & \leq & \lim_{j}[\langle y,d\rangle-\langle y,\tilde{x}_{i_{j}}\rangle-\langle x_{i_{j}},d\rangle+\|\tilde{x}_{i_{j}}\|^{2}]\nonumber \\
 & = & \lim_{j}\langle y-\tilde{x}_{i_{j}},d-\tilde{x}_{i_{j}}\rangle\leq0.\nonumber 
\end{eqnarray}
Hence $y_{0}=P_{C}(d).$ Moreover, putting $y=y_{0}$ in \eqref{eq:9.24.6p},
we get equality in the chain of inequalities, and hence 
\begin{equation}
\lim_{j}\|\tilde{x}_{i_{j}}\|^{2}=\|y_{0}\|^{2}\label{eq:9.24.8}
\end{equation}
and 
\[
\lim_{j}\langle y_{0}-\tilde{x}_{i_{j}},d-\tilde{x}_{i_{j}}\rangle=0.
\]
By \eqref{eq:9.24.3} and \eqref{eq:9.24.8}, it follows from \cite[Theorem 9.10(2)]{Deustch01}
that $\|\tilde{x}_{i_{j}}-y_{0}\|\to0$. Hence 
\begin{equation}
\|\tilde{x}_{i_{j}}-P_{C}(d)\|=\|\tilde{x}_{i_{j}}-y_{0}\|\to0.\label{eq:9.24.11}
\end{equation}
We now show an alternative strategy different from what was presented
in \cite[Theorem 9.24]{Deustch01}. Recall the definition of $v_{i}$
in \eqref{eq:def-v-i}. We have 
\[
\begin{array}{rcl}
v_{i} & = & \frac{1}{2}\|\tilde{x}_{i}-P_{C}(d)\|^{2}+\underset{l=i-m}{\overset{i}{\sum}}\langle e_{l},\tilde{x}_{l}-P_{C}(d)\rangle\\
 & = & \frac{1}{2}\|\tilde{x}_{i}-P_{C}(d)\|^{2}+\langle d-\tilde{x}_{i},\tilde{x}_{i}-P_{C}(d)\rangle+\underset{l=i-m}{\overset{i}{\sum}}\langle e_{l},\tilde{x}_{l}-\tilde{x}_{i}\rangle\\
 & \leq & \frac{1}{2}\|\tilde{x}_{i}-P_{C}(d)\|^{2}+\|d-\tilde{x}_{i}\|\|\tilde{x}_{i}-P_{C}(d)\|+\underset{l=i-m}{\overset{i}{\sum}}\langle e_{l},\tilde{x}_{l}-\tilde{x}_{i}\rangle.
\end{array}
\]
 From \eqref{eq:9.24.11} and \eqref{eq:9.23.2}, and the fact that
$v_{i}\geq0$, we have $\liminf_{i\to\infty}v_{i}=\lim_{j\to\infty}v_{i_{j}}=0$.
Since $\{v_{i}\}$ is nonincreasing, we have $\lim_{i\to\infty}v_{i}=0$.
Moreover, recall back in the proof of Lemma \ref{lem:9.20} that $v_{i}\geq\frac{1}{2}\|\tilde{x}_{i}-y_{0}\|^{2}.$
These facts combine to show us that $\tilde{x}_{i}\to y_{0}$, which
is what we seek.
\end{proof}
\bibliographystyle{amsalpha}
\bibliography{../refs}

\end{document}